\documentclass[11pt]{article}
\usepackage{amsmath, amssymb, amscd, amsthm, amsfonts}
\usepackage{mathrsfs}
\usepackage{graphicx}
\usepackage{float} 
\usepackage{subfigure}
\usepackage{hyperref}
\usepackage{enumerate}
\usepackage[top=2.5cm, bottom=2.5cm, left=2.5cm, right=2.5cm]{geometry}
\numberwithin{equation}{section}

\title{\textbf{On Natural Measures of SLE- and CLE-Related \\Random Fractals}}
\author{Gefei Cai\thanks{School of Mathematical Sciences, Peking University. Email:  \href{mailto:1900010609@pku.edu.cn}{1900010609@pku.edu.cn}}  \and Xinyi Li\thanks{Beijing International Center for Mathematical Research, Peking University. Email:  \href{mailto:xinyili@bicmr.pku.edu.cn}{xinyili@bicmr.pku.edu.cn}  Research supported by the National Key R\&D Program of China (No.\ 2021YFA1002700 and No.\ 2020YFA0712900) and NSFC (No.\ 12071012).}}
\date{\today}

\newtheorem{theorem}{Theorem}[section]
\newtheorem{definition}[theorem]{Definition}
\newtheorem{lemma}[theorem]{Lemma}
\newtheorem{claim}[theorem]{Claim}
\newtheorem{corollary}[theorem]{Corollary}
\newtheorem{proposition}[theorem]{Proposition}
\newtheorem{remark}[theorem]{Remark}

\begin{document}

\maketitle

\begin{abstract}
In this paper, we construct and then prove the up-to constants uniqueness of the natural measure on several random fractals, namely the SLE cut points, SLE boundary touching points, CLE pivotal points and the CLE carpet/gasket. As an application, we also show the equivalence between our natural measures defined in this paper (i.e.\ CLE pivotal and gasket measures) and their discrete analogs of counting measures in critical continuum planar Bernoulli percolation in [Garban-Pete-Schramm, {\it J.~Amer.~Math.~Soc.},~2013]. Although the existence and uniqueness for the natural measure for CLE carpet/gasket have already been proved in [Miller-Schoug, {arXiv:2201.01748}], in this paper we provide with a different argument via the coupling of CLE and LQG.

\end{abstract}
\section{Introduction}\label{INT}

\subsection{Background and Motivations}

The Schramm-Loewner Evolution (SLE) is a canonical conformally invariant family of random curves in a simply connected domain $D\subset \mathbb{C}$. The loop version of SLE is known as the conformal loop ensemble (CLE). In the study of two-dimensional statistical physics, SLE curves are natural candidates for the scaling limits of interfaces in many critical models, while CLE's are for ``full'' scaling limits of nested interfaces. Loops in ${\rm CLE}_\kappa$ are locally absolutely continuous with respect to ${\rm SLE}_\kappa$ curves; they are a.s.\ simple, disjoint and away from the boundary when $\kappa<4$ while they are self-intersecting and intersect with each other when $\kappa\in(4,8)$. We refer readers to e.g.\ \cite{LawlerSLE} and \cite{sheffield2011conformal} for further reference.

There are several kinds of special points on SLE and CLE which are of particular interest. The \textbf{cut points} of a non-simple SLE curve are intersections of its left and right boundaries. On a chordal non-simple SLE, points at which the SLE intersects with the boundary of the domain are its \textbf{boundary touching points}. The \textbf{pivotal points} of a ${\rm CLE}_{\kappa'}$, $\kappa'\in(4,8)$ are those points which are in the intersection of two different loops in $\Gamma$ or visited by one loop in $\Gamma$ at least twice. The \textbf{CLE carpet} or \textbf{gasket} are those points that are not surrounded by any loop of CLE, corresponding to whether the CLE is simple or not respectively. These special points also have connections with discrete statistical models. For example, in two-dimensional critical percolation, the pivotal points of ${\rm CLE}_6$ correspond to those sites which have four macroscopic alternating arms, and the gasket corresponds to sites in a macroscopic cluster (i.e.  where macroscopic one-arm events occur).

The study of natural measures on SLE-type curves already has a long history. In \cite{lawler2009natural}, the authors use a Doob-Meyer martingale decomposition to construct the natural parametrization of a simple SLE, and they show this parametrization performs like a $d$-dimensional volume measure under conformal maps. In \cite{2015}, the authors prove that the Minkowski content for SLE curves exists, and the optimal H\"older exponent under Minkowski content parametrization has been proved in \cite{zhan2018optimal}. Then in \cite{benoist2017natural}, the reverse is proved, i.e.\ a parametrization of a simple SLE is unique once it performs like a $d$-dimensional volume measure under conformal maps. It should also be mentioned that in showing the uniqueness, \cite{benoist2017natural} uses an approaching of tilting the natural measure through Liouville quantum gravity (LQG). This idea has a profound inspiration for later works in this direction.

There are also many works on the natural measure on special points of SLE and CLE. In \cite{bound2009} the authors axiomatically construct the measure on the boundary intersection points of an ordinary ${\rm SLE}_{\kappa'}$ curve, by a martingale approach; later in \cite{Minbound} its Minkowski content is showed to exist. In \cite{Zhangreen}, the author shows an estimation of Green function of cut points lying on the boundary of ${\rm SLE}_{\kappa'}$ curve, which can be viewed as the first step of proving the existence its Minkowski content. In \cite{zhan2021}, the existence of the Minkowski content of boundary intersection points of ${\rm SLE}_{\kappa'}(\rho_1,\rho_2)$ is proved. The existence of Minkowski content for cut points of general ${\rm SLE}_{\kappa'}$'s for $\kappa'\in(4,8)$ remains open except in a special case, where the work \cite{HLLSBrown} shows the existence of Minkowski content of Brownian cut points as well as for cut points of ${\rm SLE}_6$. Recently, \cite{miller2021cle} discusses the axiomatic construction on the natural measures of CLE carpet and gasket. 

The natural measures also come from the scaling limits of counting measures on special discrete objects in discrete models. In \cite{garban2014pivotal}, it is proved that in critical site percolation on the triangular lattice, the counting measures on sites of 1-arm or 4-arms converges, after proper normalization as mesh size tends to 0. As the full scaling limit of interfaces of critical site percolation is ${\rm CLE}_6$, it is natural to expect that these limiting measures should be the natural measures on ${\rm CLE}_6$ gasket and pivotal points respectively. In a similar fashion, in \cite{camia2017conformal}, the scaling limit cluster measures of the critical planar Ising model and FK-Ising model are showed to exist. However, their respective conformal covariance properties have not been fully proved yet (see Theorem 2.4 in \cite{camia2017conformal}).

In this paper we will axiomatically construct the natural measures supported on SLE cut points, SLE boundary touching points, CLE pivotal points, CLE carpet and gasket. In each case, we will define its natural measure by three axioms, and then prove that measures satisfying these axioms exist and must be unique up to a constant. Here we use the word \textbf{natural} because the axioms through which they are defined are properties a natural\footnote{in the sense that Lebesgue measure is natural for the Euclidean space $\mathbb{R}^d$} measure on these fractals (e.g., the Minkowski content or scaling limit of the counting measure of corresponding discrete models, should they exist) should satisfy; for further discussions, see Section \ref{sec:discussion}.

\subsection{Main Results}
We first study the natural measure on boundary touching points of SLE.  Suppose $\eta$ is an ${\rm SLE}_\kappa(\rho)$ process in $\mathbb{H}$ from $0$ to $\infty$ with a single force point located at $0+$, where $\kappa\in(0,4)$ and $\rho\in\left(-2,\frac{\kappa}{2}-2\right)$.

\begin{definition}\label{def2}
A Radon measure $\nu_0=\nu_0(dz;\eta)$ is called the natural measure on SLE boundary touching points, if it is supported on $\eta\cap\mathbb{R}_+$, measurable with respect to $\eta$ and satisfies the followings:
\begin{enumerate}[1)]
\item \textbf{Conformal Markov property}. For any stopping time $t>0$ such that $\eta(t)\in\mathbb{R}_+$, conditioned on $\eta[0,t]$, the joint law of $(\phi_t^{-1}(\eta), |\phi_t'|^{-d} \nu_0\circ\phi_t)$ is equal to the original joint law of $(\eta,\nu_0)$, where $\phi_t$ is any conformal map from the unbounded component of $\mathbb{H}\backslash\eta[0,t]$ to $\mathbb{H}$ with $\phi(\eta(t))=0$ and $\phi(\infty)=\infty$, and $d:=\frac{(\rho+4)(\kappa-4-2\rho)}{2\kappa}$ is the Hausdorff dimension of $\eta\cap\mathbb{R}_+$. 
\item \textbf{Scaling}. For any scaling map $\phi:z\mapsto rz$ ($r>0$), the Radon-Nikodym derivative between $\nu_0(\phi(\cdot);\phi(\eta))$ and $\nu_0(\cdot;\eta)$ is
\begin{equation*}
\frac{\mathrm{d}\nu_0(\phi(\cdot);\phi(\eta))}{\mathrm{d}\nu_0(\cdot;\eta)}=r^d.
\end{equation*}
\item \textbf{Finite expectation}. For any interval $[0,c]$, $\mathbb{E}[\nu_0([0,c];\eta)]<\infty$.
\end{enumerate}
\end{definition}

\begin{theorem}\label{thm2}
The natural measure on SLE boundary touching points exists, and is uniquely determined up to a deterministic multiplicative constant.
\end{theorem}

We now turn to CLE pivotal points. Suppose $D$ is a simply connected domain or the whole-plane, and $\Gamma$ is a ${\rm CLE}_{\kappa'}$ configuration in it.

\begin{definition}\label{def3}
A measure $\nu_0(\cdot;D,\Gamma)$ (it should be thought as a measure-valued function with argument $D$ and $\Gamma$)  is called the natural measure on CLE pivotal points if it is supported on the pivotal points on $\Gamma$, a.e. $\sigma$-finite and satisfies the followings:
\begin{enumerate}[1)]
\item \textbf{Coordinate change formula}. For a conformal transformation $\psi:D\to D'$, the Radon-Nikodym derivative between $\nu_0(\psi(\cdot);D',\psi(\Gamma))$ and $\nu_0(\cdot;D,\Gamma)$ is
\begin{equation*}
\frac{\mathrm{d}\nu_0(\psi(\cdot);D',\psi(\Gamma))}{\mathrm{d}\nu_0(\cdot;D,\Gamma)}=|\psi'(x)|^d,
\end{equation*}
 where $d:=2-\frac{(12-\kappa')(4+\kappa')}{8\kappa'}$ is the Hausdorff dimension of ${\rm CLE}_{\kappa'}$ pivotal points.
\item \textbf{Locality}. For given $D$ and configuration $\Gamma$, $\nu_0(U;D,\Gamma)$ is determined by $\{l\subset U: l \in\Gamma\}$ for any open set $U\subset D$.
\item \textbf{Finite Expectation on the Macroscopic Pivotal Points on Pseudo-Interface}. This axiom will be precisely stated at the end of Section \ref{BC}.
\end{enumerate}
\end{definition}

\begin{remark}
It can be seen in the following that the CLE pivotal measure we define poses no mass on any deterministic smooth curve but infinite mass on any open set, therefore we need a carefully chosen finiteness criterion. One reasonable choice is the finiteness on the intersection of two interfaces, which gives rise to the \textbf{pseudo-interface} in the third axiom.
\end{remark}

\begin{remark}
We now heuristically explain what a pseudo-interface for whole-plane CLE is. For CLE in a bounded Jordan domain, one can perform a radial exploration from any boundary point to an inner point to obtain a radial interface. However in the whole-plane case, it is not evident how to define a radial exploration from some point (say, the origin) to infinity properly. So instead, we choose a CLE loop, radially explore inward and outward respectively from somewhere on the loop and then concatenate this loop and two exploration curves. This curve is what we call the pseudo-interface; see Section \ref{BC} for the precise definition.
\end{remark}

\begin{theorem}\label{thm3}
For any $\kappa'\in(4,8)$, the natural measure on ${\rm CLE}_{\kappa'}$ pivotal points exists, and is uniquely determined up to a deterministic multiplicative constant.
\end{theorem}

We next study the natural measure on SLE cut points. We fix $\kappa'\in(4,8)$ and suppose $\eta$ is a ${\rm SLE}_{\kappa'}$ curve on a simply connected domain $(D,a,b)$.

\begin{definition}\label{def0}
A measure $\nu_0(\cdot;D,\eta)$ 
is called the natural measure on SLE cut points if it is supported on the cut points on $\eta$ and satisfies the followings:
\begin{enumerate}[1)]
\item \textbf{Coordinate change formula}. For a conformal transformation $\psi:D\to D'$, the Radon-Nikodym derivative between $\nu_0(\psi(\cdot);D',\psi(\eta))$ and $\nu_0(\cdot;D,\eta)$ is
\begin{equation*}
\frac{\mathrm{d}\nu_0(\psi(\cdot);D',\psi(\eta))}{\mathrm{d}\nu_0(\cdot;D,\eta)}=|\psi'(x)|^d,
\end{equation*}
 where $d:=3-\frac{3\kappa'}{8}$ is the Hausdorff dimension of the cut point set. 
\item \textbf{Locality}. For given $D$ and configuration $\eta$, $\nu_0(U;D,\eta)$ is determined by the segment $\eta\cap U$ for any open set $U\subset D$.
\item \textbf{Finite expectation}. For any compact subset $K$, $\mathbb{E}[\nu_0(K;D,\eta)]<\infty$.
\end{enumerate}
\end{definition}

\begin{theorem}\label{thm0}
The natural measure on SLE cut points exists, and is uniquely determined up to a deterministic multiplicative constant.
\end{theorem}

Finally, we turn to the CLE carpet and gasket. We keep the notation as in the case of CLE pivotal points. 

\begin{definition}\label{def1}
Suppose $\kappa\in(\frac{8}{3},4)$. A Radon measure $\nu_0(\cdot;D,\Gamma)$ 
is called the natural measure on CLE carpet if it is supported on $\Gamma$ and satisfies the followings:
\begin{enumerate}[1)]
\item \textbf{Coordinate change formula}. For a conformal transformation $\psi:D\to D'$, the Radon-Nikodym derivative between $\nu_0(\psi(\cdot);D',\psi(\Gamma))$ and $\nu_0(\cdot;D,\Gamma)$ is
\begin{equation*}
\frac{\mathrm{d}\nu_0(\psi(\cdot);D',\psi(\Gamma))}{\mathrm{d}\nu_0(\cdot;D,\Gamma)}=|\psi'(x)|^d,
\end{equation*}
 where $d=2-\frac{(3\kappa-8)(8-\kappa)}{32\kappa}$ is the Hausdorff dimension of the ${\rm CLE}_\kappa$ carpet;
\item \textbf{Locality}. For given $D$ and configuration $\Gamma$, $\nu_0(U;D,\Gamma)$ is determined by the local loop configuration $\{l\in\Gamma:l\subset U\}$ for any open set $U\subset D$.
\item \textbf{Finite Expectation}. For each compact set $K\subset D$, $\mathbb{E}[\nu_0(K;D,\Gamma)]<\infty$.
\end{enumerate}
We can define the natural measure on CLE gasket in parallel. We only need to replace $\kappa,d$ by $\kappa'\in(4,8)$ and $d'=2-\frac{(3\kappa'-8)(8-\kappa')}{32\kappa'}$ in the definition above, respectively.
\end{definition}

\begin{theorem}\label{thm1}
The natural measure on CLE carpet or gasket exists and is uniquely determined up to a deterministic multiplicative constant. 
\end{theorem}

\subsection{Discussions}\label{sec:discussion}

In this paper we deal with those special points on SLE and CLE by using an LQG approach. In our proof of the above theorems, we will first construct the natural measures (some of which has already appeared in the literature), and then use a LQG approach to show the uniqueness. Our main idea of uniqueness is to introduce an independent LQG background, make use of the conformal invariant property of quantized version of the natural measure, and exploit the nice structure of the exploration processes. For example, these exploration process will have the law of a stable subordinator when coupling with LQG, which possesses a rather good scaling property. A key identity that connects the LQG and Euclidean world is the KPZ relation; see Remark \ref{KPZ} or Appendix B in \cite{duplantier2020liouville} for further reference. 

We alnote that a recent paper \cite{miller2021cle} also discusses the axiomatic construction on CLE carpet and gasket. Their axioms are slightly different from our Definition \ref{def1}, where they use the Markovian property to replace the locality. Therefore, though the constructions of natural measures in \cite{miller2021cle} and this paper are similar, the approaches to show the uniqueness are rather different: in \cite{miller2021cle} they use a martingale argument, while we make heavy use of LQG. It is also worth mentioning that although locality is a priori stronger than the Markovian property as an assumption, it is indeed satisfied by all natural measures (or candidates thereof) for these fractals defined via different approaches. Moreover, the choice of locality as an axiom gives us a unified treatment towards various types of fractals in this paper. 

As the construction of quantized measure usually requires the \textbf{finiteness of energy} which does not appear in our axioms of natural measures, we will use an approach different from the usual construction via Gaussian multiplicative chaos (GMC), that is, to take the quantum natural measure we have already known into consideration; see Section \ref{ANO} for more details. We also remark that our proof can not be easily extended to the critical case $\kappa=4$ for the $\rm CLE_{4}$ carpet as the corresponding parameter $\gamma(=\sqrt\kappa=2)$ of LQG is critical. In \cite{miller2021cle}, an approximation approach is offered for the $\rm CLE_{4}$ carpet measure.

In connection with corresponding discrete models, in Section \ref{PC1} and \ref{PC2}, we will show that the natural measure on $\rm CLE_6$ pivotal points and gasket in this paper is up-to-constants equal to the pivotal and area measure in critical planar continuum Bernoulli percolation constructed in \cite{garban2014pivotal}, by checking that the latter measures satisfy Definitions \ref{def3} and \ref{def1} respectively. However, for the cluster measure of Ising model, we can only conjecture that it should be natural measure on ${\rm CLE}_3$ carpet since its conformal covariance have not been proved yet (see \cite{camia2017conformal} for further reference).

\textbf{Open questions}. The first question is, if the Minkowski contents of those special random fractals exists. If so, they should satisfy the axioms we pose in this paper and offer a concrete construction of the natural measure for these fractals. Another question is that, as we choose to avoid using the usual GMC approach to obtain a quantized measure, we wonder whether the natural measures we construct have finite $d$-dimensional energy. Although we conjecture this is true for all the 
fractals we investigate in this paper, it seems that a proof is not very easy since one needs to know how these random fractals are embedded into the Euclidean world. However, if the existence of their Minkowski contents were proved, its finiteness of energy would be easy to check.

\textbf{Outline of the paper}. This paper is structured as follows. In Section \ref{PRE}, we review some background on LQG, quantum surfaces and their conformal weldings. In Section \ref{ANO} we show an another approach for quantized measure of LQG rather than directly constructing its Gaussian multiplicative chaos in the classical literature. Then in the following four sections, we will prove our main theorems \ref{thm2} to \ref{thm0}. The In Section \ref{PVT} and \ref{GAT}, we will also discuss the pivotal and area measures of critical percolation respectively.

\textbf{Acknowledgment}. The authors thank Xin Sun for suggesting this problem and helpful discussions during various stages of this project.

\section{Preliminaries}\label{PRE}

\subsection{Schramm-Loewner Evolution and Conformal Loop Ensemble}

The Schramm-Loewner evolution (SLE) is a random fractal curve first introduced by Schramm in 1999 \cite{Sch00} to describe the scaling limits of interfaces in two-dimensional discrete models from statistical mechanics. In this work, we will be constantly using the following generalization named the ${\rm SLE}_\kappa(\rho)$ process of the originial SLE process. Generally, for $\underline{x}_L= (x^{k,L} < ... < x^{1,L})$
where $x^{1,L}\le 0$ and $\underline{x}_R= (x^{1,R} < ... < x^{k,R})$ where $x^{1,R}\ge 0$, consider force points $(\underline{x}_L,\underline{x}_R)$ on the real line corresponding to the weights $\underline{\rho}=(\underline{\rho}_L,\underline{\rho}_R)$ with $\underline{\rho}_L= (\rho^{k,L} < ... < \rho^{1,L}),\underline{\rho}_R= (\rho^{1,R} < ... < \rho^{k,R})$ such that each $x^{i,q}$ corresponds to the weight $\rho^{i,q}$. An ${\rm SLE}_{\kappa}(\underline\rho)$ process is the measure on continuously growing compact hulls $K_t$, such that the conformal maps $g_t:\mathbb{H}\backslash K_t\to \mathbb{H}$ with $|g_t(z)-z|\to0,z\to\infty$ satisfy the Loewner equation $\partial_t g_t(z)=\frac{2}{g_t(z)-W_t}$ with $(W_t)$ the solution of the following stochastic integral equations
\begin{equation*}
\begin{aligned}
W_t&=\sqrt{\kappa}B_t+\sum\limits_{q\in\{L,R\}}\sum\limits_i \int_{0}^t \frac{\rho^{i,q}}{W_s-V_s^{i,q}}ds,\ 
V_t^{i,q}=\int_0^t \frac{2}{V_s^{i,q}-W_s}ds+x^{i,q},\ q\in\{L,R\}
\end{aligned}
\end{equation*}
where $(B_t)$ is a standard one-dimensional Brownian motion. One can prove the uniqueness of the solution, and can further show that the hulls $(K_t)$ can be generated by a unique curve $\eta$ from $0$ to $\infty$ in $\mathbb{H}$. We also call this $\eta$ an ${\rm SLE}_\kappa(\underline\rho)$ curve, and we call $(f_t)\triangleq(g_t-W_t)$ its \textit{centered} Loewner flow. Note that the ordinary ${\rm SLE}_\kappa$ is the case of no force points. In particular, if there is only one force point locating at  $0$, we would denote the corresponding law by ${\rm SLE}_\kappa(\rho)$; and if there are two in $(\underline{x}_L,\underline{x}_R)=(0-,0+)$ with their wights $\rho=(\rho_1,\rho_2)$, we simply denote the corresponding SLE curve by ${\rm SLE}_\kappa(\rho_1,\rho_2)$ in the following.

The conformal loop ensemble (CLE) is a canonical conformally invariant probability measure on infinite collections of non-crossing loops with an index $\kappa\in(\frac{8}{3},8)$, where each loop in the collection is locally like an ordinary ${\rm SLE}_\kappa$ curve. For $\kappa\in(\frac{8}{3},4]$, according to Theorem 1.1 to 1.4 in \cite{sheffield2011conformal}, ${\rm CLE}_\kappa$ on a simply connected domain can be characterized by the domain Markov property and conformal invariance, and loops in the configuration are simple and do not intersect with each other. For $\kappa'\in(4,8)$, ${\rm CLE}_{\kappa'}$ can be constructed by the ${\rm SLE}_{\kappa'}(\kappa'-6)$ exploration tree (see Section 4.3 in \cite{sheffield2009exploration}), and in this case loops are non-simple, self-touching and can have intersections. We refer readers to \cite{sheffield2011conformal} and \cite{sheffield2009exploration} for more on CLE.

\subsection{Liouville Quantum Gravity}\label{LQG}

In this subsection we quickly review the theory of LQG (one can see \cite{berestycki2015introduction} for further reference). We first consider the Gaussian Free Field (GFF) on a simply connected domain $D\subset\mathbb{C}$. For two functions $f,g\in C_0^\infty (D)$, define their Dirichlet inner product by $(f,g)_\nabla=\frac{1}{2\pi}\int_D \nabla f(x) \cdot \nabla g(x) dx$. Let $H(D)$ be the Hilbert space closure of $\{f:f\in C_0^\infty(D) \text{ and } \int_D f dz=0\}$ with respect to the Dirichlet inner product. The GFF $h$ on $D$ can be expressed as a random linear combination of a orthonormal basis $(\phi_i)$ of $H(D)$, that is, $h=\sum\limits_i \alpha_i \phi_i$ where $(\alpha_i)$ are i.i.d. standard Gaussian random variables.

The goal of Liouville quantum gravity (LQG) is the study of quantum surfaces, i.e.\ of domains carrying Liouville measure structures. From a field $h'$ and a measure $\sigma$, one can build the quantized measure $:e^{\gamma h'}\sigma:$ by proper normalization. Let $\sigma$ be a Radon measure such that for some $d>0$,
\begin{equation*}
\iint_{\bar D\times \bar D}\frac{\sigma(dx)\sigma(dy)}{|x-y|^{d-\varepsilon}}<\infty,\forall \varepsilon\in(0,d)
\end{equation*}
(this condition is called the \textbf{finiteness of $d$-dimensional energy}). Suppose $h'=h+m$ is the sum of a continuous function $m$ and Gaussian field $h$ with a Gaussian kernel $K(x,y)=\log |x-y|+g(x,y)$ where $g$ is a continuous function on $\bar D\times\bar D$. Then according to \cite{berestycki2017elementary}, for values of the parameter $\gamma\in(0,\sqrt{2d})$, one can build the quantized measure $:e^{\gamma h'}\sigma:$ by taking the weak limit in probability of $e^{\gamma h_\varepsilon'-\frac{1}{2}\gamma^2E h_\varepsilon'^2}\sigma$, where $h'_\varepsilon=h'\ast\theta_\varepsilon$ is the mollification of $h'$ with a bump function $\theta$ on $\mathbb{C}$ and $\theta_\varepsilon(\cdot)=\theta(\cdot/\varepsilon)$.

In particular, when $\sigma=dz$ is the Lebesgue measure, we have the regularization $e^{\gamma h'} dz=\lim\limits_{\varepsilon\to 0}\varepsilon^{\gamma^2/2}e^{\gamma h'_\varepsilon}dz$ where $h'_\varepsilon(z)$ can be taken as the mean value of $h'$ on the circle $\partial B(z,\varepsilon),z\in D$. In the case of $\sigma=dx$ that is the Lebesgue measure on the boundary $\partial D$, we have a similar regularization $e^{\gamma h'} dx=\lim\limits_{\varepsilon\to 0}\varepsilon^{\gamma^2/4}e^{\gamma h'_\varepsilon/2}dx$ where $h'_\varepsilon(x)$ can be taken as the mean value of $h'$ on $\partial B(x,\varepsilon)\cap D,x\in \partial D$. These two measures are called the LQG area and length measure respectively.

The following proposition from Proposition 2.13 in \cite{benoist2017natural} is an inverse of this procedure. This notion will be used frequently when we construct natural measures from their \textit{quantum} counterparts.

\begin{proposition}\label{REV}
Suppose $h'$ is a Gaussian field, and $\sigma$ is a Radon measure supported on $D$. We can recover the measure $\sigma$ from its quantum counterpart, namely the LQG tilting measure $e^{\gamma h'}\sigma$ as $\sigma(dz)=e^{-\frac{\gamma^2}{2}\hat K(z)}\mathbf{E}[e^{\gamma h'}\sigma(dz)]$, where $\hat K(z)=\lim\limits_{\varepsilon\to 0}\log\varepsilon+\mathrm{Var}(\theta_z^\varepsilon,h')$.
\end{proposition}

Let $\textbf{P}_D^{(\gamma,z)}(dh)$ denote the laws of $h$ and $h+\gamma K(\cdot,z)$, viewed as probability measures on $H^{-1}(D)$.

\begin{lemma}\label{ADD}
In the topology of vague convergence of measures, we have
\begin{equation*}
\lim\limits_{\varepsilon\to0}e^{-\frac{\gamma^2}{2}\mathrm{Var}(\theta_z^\varepsilon,h)}e^{\gamma h_\varepsilon(z)}\mathbf{P}_D(dh)=\mathbf{P}_D^{(\gamma,z)}(dh)
\end{equation*}
\end{lemma}
\begin{proof}
By Girsanov transform we have $e^{-\frac{\gamma^2}{2}\mathrm{Var}(\theta_z^\varepsilon,h)}\textbf{E}_D[e^{\gamma h_\varepsilon(z)}f(h)]=\textbf{E}_D[f(h+\gamma K_\varepsilon(\cdot,z))]$, where $K_\varepsilon(w,z)$ is the average of $K(w,\cdot)$ on $\partial B_\varepsilon(z)$. The results follows since $K_\varepsilon(\cdot,z)\to K(\cdot, z), \varepsilon\to0$ in $H^{-1}(D)$.
\end{proof}

In the end we would like to emphasis that the interest of LQG is the Liouville measure structure while the background field $h'$ only serves as a tool to construct these measures in fixed coordinates, and hence should be changed appropriately when we change coordinates, so that the geometric objects are kept unchanged. The Liouville coordinate change formula is given by $h_{\phi(D)}=h_D\circ\phi^{-1}+Q\log|(\phi^{-1})'|$, where $Q=\gamma/2+2/\gamma$. The following proposition is from in Section 5.5 in \cite{berestycki2015introduction}. 

\begin{proposition}\label{LCC}
Liouville bulk and boundary measures are then invariant under this change of coordinates.
\end{proposition}

In accordance with the above proposition, we define the \textbf{quantum surface} as a class of field-carrying complex domains $(D,h)$ modulo Liouville changes of coordinates. A representative of the equivalence class is called an \textbf{embedding} of a quantum surface. We will always choose the \textbf{circle average embedding} in the followings, i.e. in the half-plane $\mathbb{H}$ we choose the embedding such that $\sup\{s>0:h_s(0)+Q\log s=0\}=1$.

\subsection{Quantum Wedges, Half-Plane and Disks}\label{QW}
In this subsection we briefly introduce several useful variants of GFF, namely the quantum wedge and quantum disk. According to the Liouville coordinate change formula and the Riemann mapping theorem, we only need to deal with several special domains. Note that we have the radial-lateral decomposition $H(\mathbb{H})=H_1(\mathbb{H})\oplus H_2(\mathbb{H})$, where $H_1(\mathbb{H})$(resp.\ $H_2(\mathbb{H})$) is the subspace of $H(\mathbb{H})$ of functions which are constant (resp.\ have mean zero) on the semicircle $\{z\in\mathbb{H}:|z|=R\}$ for each $R>0$. In this case the projection of a GFF $h$ on $\mathbb{H}$ onto $H_1(\mathbb{H})$ has the distribution of $(B_{2t})_{t\in\mathbb{R}}$, where $(B_t)$ is a standard two-sided Brownian motion.

According to Section 4.2 and 4.4 in \cite{duplantier2020liouville}, we can define an $\alpha$-(thick) quantum wedges (when parameterized by $\mathbb{H}$) by specifying separately its averages on every semicircle around $0$ and what is left when we subtract these from the wedge.

\begin{definition}
Fix $\alpha<Q$ (recall $Q=\gamma/2+2/\gamma$). Suppose $(B_t)_{t\ge 0}$ is a standard Brownian motion with $B_0=0$, and $(\hat B)_{t\ge 0}$ independent of $B$ is a standard Brownian motion starting at $0$ conditioned on $\hat B_{2t}+(Q-\alpha)t>0$ for all $t>0$.  Let $(A_s)$ to be the process such that $A_s=B_{2s}+\alpha s$ for $s>0$ and $A_s=\hat B_{-2s}+\alpha s$ for $s<0$. Let $H_1(\mathbb{H})$ and $H_2(\mathbb{H})$ be as mentioned above. An $\alpha$-\textbf{(thick) quantum wedge} on $\mathbb{H}$ is the quantum surface $h'$ (when parameterized by $\mathbb{H}$ with marked points $-\infty$ and $+\infty$) such that its projection onto $H_1(\mathbb{H})$ is the function whose common value on $\partial B(0,e^{-s})\cap \mathbb{H}$ is $A_s$ for each $s\in\mathbb{R}$ and its projection onto $H_2(\mathbb{H})$ is the corresponding projection of an independent Dirichlet GFF on $\mathbb{H}$.
\end{definition}

Define the weight of a quantum wedge $W$ as
\begin{equation*}
W=\gamma\left(\frac{\gamma}{2}+Q-\alpha\right)
\end{equation*}
For $\gamma\in(0,2)$. The case of $W=2$ is referred to as the \textbf{quantum half-plane}. As explained in Section 2.1 of \cite{miller2020simple}, roughly speaking, this is the case where the marked boundary point $0$ on the half-plane is not particularly special, in the sense that resampling a boundary point $z$ according to the LQG boundary length will not change the law of quantum surface with two marked points.

\textbf{Quantum cone} is an infinite volume surface without boundary and homeomorphic to $\mathbb{C}$. Its definition is parallel to the thick quantum wedge (see Defition 4.10 in \cite{duplantier2020liouville} for a precise definition of quantum cone).

Now we will define the $\alpha$-(thin) quantum wedge for $\alpha\in(Q,Q+\gamma/2)$ using Bessel processes.

\begin{definition}
Fix $\alpha\in(Q,Q+\gamma/2)$. An $\alpha$-\textbf{(thin) quantum wedge} (when parametrized by $\mathcal{S}=\mathbb{R}\times[0,\pi]$) is defined as follows. Let $\delta=1+2(Q-\alpha)/\gamma$, and sample a $\delta$-dimensional Bessel process $Y$ starting from $0$. Let $H_1(\mathcal{S})$, $H_2(\mathcal{S})$ be the subspaces of $H(\mathcal{S})$ of functions which are constant (resp.\ have mean zero) on each vertical line $x+[0,i\pi]$ with $x\in\mathbb{R}$. For each excursion $e$ of $Y$ from $0$, we independently sample a distribution $(h_e)$ on $\mathcal{S}$ such that its projection onto $H_1(\mathcal{S})$ is given by $\frac{2\log e}{\gamma}$ (after reparametrizing such that its quadratic variation is $2dt$) and its projection onto $H_2(\mathcal{S})$ is the corresponding projection of an independent Dirichlet GFF on $\mathcal{S}$.
\end{definition}


By definition, we see that a thin quantum wedge consists of an infinite sequence of beads. As explained in Section 4.4 of \cite{duplantier2020liouville}, this sequence of beads can be viewed as a Poisson point process.

For each $\gamma\in(\sqrt 2,2)$, there is a special thin quantum wedge of weight $W=\gamma^2-2$, a bead of which we refer to as a \textbf{quantum disk} (so this quantum wedge is the concatenation of a Poisson point process of quantum disks). As in the case of quantum half-plane, roughly speaking, the case of quantum disk is that the two marked boundary points are not particularly special boundary points, in the sense that when resampling two boundary points according to the LQG boundary length the law of the two-point marked quantum surface remains unchanged. This is also mentioned in Section 2.1 of \cite{miller2020simple}.

One can also define quantum disks of general weight $W\in\mathbb{R}_+$, see Definitions 2.1 and 3.5 in \cite{SLEwelding}, and we denote the measure by $\mathcal{M}^{\text{disk}}_2(W)$ (where "2" denotes that there are two marked points). Similar to the case of quantum wedge, when $W\ge\gamma^2/2$, the sample of $\mathcal{M}^{\text{disk}}_2(W)$ is thick (i.e. homeomorphic to the unit disk $\mathbb{D}$). Otherwise for $W\in(0,\gamma^2/2)$, one can sample $\mathcal{M}^{\text{disk}}_2(W)$ by first sampling $T$ from $\left(1-\frac{2}{\gamma^2}W\right)^{-2}\text{Leb}_{\mathbb{R}_+}$ and then concatenating a sample $(u,\mathcal{D}_u)$ of a Poisson point process with intensity $1_{t\in[0,T]}\times\mathcal{M}^{\text{disk}}_2(\gamma^2-W)$ ordered by the value of $u$.

In the following, we use the notation $\textbf{QW}_D,\textbf{QD}_D$ or $\textbf{QC}$ to denote the law of quantum wedge and quantum disk on a domain $D$, or a quantum cone on the whole-plane. As explained in Section 4.5 of \cite{duplantier2020liouville}, it is possible to decompose the quantum disk measure according to its total boundary length, that is, we can define the unit boundary length quantum disk as the quantum disk conditioned on its boundary length being $1$. By scaling then we are able to define a probability measure on quantum disks with any prescribed boundary length.

\subsection{SLE-Decorated Quantum Surfaces}

In this section we cite some basic results on SLE-decorated quantum surfaces developed through the theory of mating of trees in \cite{duplantier2020liouville}, which we will use frequently in the proofs. The following theorem is Theorem 1.2 in \cite{duplantier2020liouville}.

\begin{theorem}\label{BT}
Fix $\gamma\in(0,2)$, $\kappa=\gamma^2$ and $\rho_1,\rho_2>-2$. Let $\mathcal{W}=(\mathbb{H},h,0,\infty)$  be a quantum wedge of weight $(\rho_1+\rho_2+4)$ and let $\eta$ be an ${\rm SLE}_{\kappa}(\rho_1,\rho_2)$ process in $\mathbb{H}$ from $0$ to $\infty$ with force points located at $0-,0+$ which is independent of $\mathcal{W}$. Denote $D_1$ and $D_2$ for the left and right regions of $\mathbb{H}\backslash\eta$. Then the two quantum surfaces $\mathcal{W}_1=(D_1,h,0,\infty)$ and $\mathcal{W}_2=(D_2,h,0,\infty)$ are independent quantum wedges of weight $W_i=\rho_i+2$.
\end{theorem}

An important consequence of the conformal welding theory is that it leads to another notion of time parametrization for an ${\rm SLE}_{\kappa'}$ type process $\eta'$ for $\kappa'\in(4,8)$, under which the coupling of SLE and LQG has a conformal Markov property. It is called the \textit{quantum natural time}. The following theorem is Theorem 1.18 in \cite{duplantier2020liouville}.

\begin{theorem}\label{QNT}
Fix $\gamma\in(\sqrt 2,2)$ and let $\kappa'=16/\gamma^2\in(4,8)$. Let $\mathcal{W}=(\mathbb{H},h',0,\infty)$  be a quantum wedge of weight $W=\frac{3\gamma^2}{2}-2$ and let $\eta'$ be an ${\rm SLE}_{\kappa'}$ process in $\mathbb{H}$ from $0$ to $\infty$ which is independent from $h'$. Then there is a random function $q_u$ (which is called the quantum natural time of the curve $\eta'$), such that if $(f_t)$ denotes the centered Loewner flow of $\eta'$ with the capacity time parameterization, then viewing the pair $(h,\eta)$ as path-decorated quantum surfaces we have that
\begin{equation*}
(h,\eta')\overset{d}=(h\circ f_{q_u}^{-1}+Q\log|(f_{q_u}^{-1})'|,f_{q_u}(\eta')),\forall u>0.
\end{equation*}
Therefore the above claim could also be stated as that the joint law of $(h',\eta')$ is invariant under the operation of cutting along $\eta'$ until a given quantum natural time and then conformally mapping back and applying the Liouville coordinate change formula (see Proposition \ref{LCC}).
\end{theorem}

For general ${\rm SLE}_{\kappa'}(\rho_1,\rho_2)$ curves, we have the following result (see Theorem 1.16 in \cite{duplantier2020liouville}).

\begin{theorem}\label{WD}
Fix $\gamma\in(\sqrt 2,2)$ and $\kappa'=16/\gamma^2\in(4,8)$. Fix $\rho_1,\rho_2\ge \frac{\kappa'}{2}-4$, and let $W_i=\gamma^2-2+\frac{\gamma^2}{2}\rho_i$ for $i=1,2$ and $W=W_1+W_2+2-\frac{\gamma^2}{2}$. Suppose $W\ge \frac{\gamma^2}{2}$. Let $\mathcal{W}=(\mathbb{H},h',0,\infty)$ be a quantum wedge of weight $W$ and let $\eta'$ be an independent ${\rm SLE}_{\kappa'}(\rho_1,\rho_2)$ process in $\mathbb{H}$ from $0$ to $\infty$ with force points at $0-,0+$. Denote $\mathcal{W}_1$ (resp.\ $\mathcal{W}_2$) for the quantum surfaces formed by those components of $\mathbb{H}\backslash\eta'$ which are to the left (resp. right) of $\eta'$, and denote $\mathcal{W}_3$ for the quantum surface between the left and right boundaries of $\eta'$. Then the quantum surfaces $\mathcal{W}_1,\mathcal{W}_2,\mathcal{W}_3$ are three independent quantum wedges of weight $W_1$, $W_2$ and $(2-\frac{\gamma^2}{2})$ respectively.
\end{theorem}

\begin{remark}
The quantum natural time for ${\rm SLE}_{\kappa'}$ type curve can also be characterized by the fact that it is the parametrization under which the left and right the boundary length processes turn out to be independent stable Levy processes. We also mention that this quantum natural time is the analog of natural parametrization for simple ${\rm SLE}_\kappa$ by its $\gamma$-LQG length for $\kappa\in(0,4)$ or for space-filling ${\rm SLE}_{\kappa'}$ curve by its $\gamma$-LQG area for $\kappa'\ge8$.
\end{remark}

\subsection{Stable Subordinators}

In the proofs of the paper, we will often deal with stable subordinators. For each $\beta>0$, a Levy process $(\tau_t)_{t\ge 0}$ is called a $\beta$-stable subordinator if $\tau$ is a.s. increasing and $\tau_{at}\overset{d}=a^{1/\beta}\tau_t$ for each $a>0$. We can exactly calculate the Laplace transform of this stable subordinator, which is a special case of Levy-Khintchine formula (one can refer to discussions below Theorem 1.2 in \cite{inbook} for example).

\begin{lemma}\label{LT}
For a $\beta$-stable subordinator $(\tau_t)$, there is some constant $c>0$ such that its Laplace transform $Ee^{-\lambda\tau_t}$ is equal to $e^{-ct\lambda^{\beta}}$ for all $t\ge 0$. Furthermore, the law of a $\beta$-stable subordinator is unique up to a constant (as a stochastic process).
\end{lemma}
\begin{proof}
Since $(\tau_t)$ has independent and stationary increments, we can see $Ee^{-\lambda \tau_{t+s}}=Ee^{-\lambda \tau_t}Ee^{-\lambda \tau_s}$. Since $t\mapsto Ee^{-\lambda t}$ is monotone, there is a number $c=c(\lambda)$ such that $Ee^{-\lambda \tau_{t}}=e^{-c(\lambda)t},\forall t\ge0$. Furthermore by scaling property $\tau_{at}\overset{d}=a^{1/\beta}\tau_t$, we have $a^{\beta}c(\lambda)=c(a\lambda)$ for any $a>0$, therefore $c(\lambda)=c(1) \lambda^{\beta}$. Since $\tau_t$ has independent and stationary increments, we can calculate the joint distribution similarly.
\end{proof}

We emphasize that there is a trivial case $\beta=1$.
\begin{corollary}\label{TVI}
For a stable subordinator $(\tau_t)$ with index $1$, there is a deterministic constant $c$ such that almost surely $\tau_t=ct,\forall t\ge 0$.
\end{corollary}

For a subordinator $\tau$, the range $\mathcal{R}_\tau$ of $\tau$ is defined as the closure of $\{\tau_t:t\ge0\}$. Denote $m_\tau$ for the pushforward of Lebesgue measure on $[0,\infty)$ by $\tau$. We call it the local time on $\mathcal{R}_\tau$ since it is a measure supported on $\mathcal{R}_\tau$. The following lemma is Lemma 5.13 in \cite{holden2021convergence}, which we will heavily rely on in the following sections.
\begin{proposition}
Suppose $(\tau_t)$ is a $\beta$-stable subordinator. Then almost surely, the $\beta$-occupation measure $m_{\mathcal{R}_\tau}$ of $\mathcal{R}_\tau$ is well-defined, and there exists a deterministic constant $c=c(\beta)>0$ such that $m_\tau[0,t]=m_{\mathcal{R}_\tau}[0,t]$ for all $t>0$.
\end{proposition}

We will also record the following finiteness of moments result for $m_{\mathcal{R}_\tau}$.
\begin{proposition}\label{OCU}
For a $\beta$-stable subordinator $(\tau_t)$ with $\tau_0=0$, the moment $E[m_\tau[0,1]^p]$ and $E[m_{\mathcal{R}_\tau}[0,1]^p]$ are finite for any $p>0$.
\end{proposition}
\begin{proof}
Since $m_\tau[0,1]=\tau^{-1}(1)$, we observe that $\{m_\tau[0,1]>s\}=\{\tau_s<1\}$. By Lemma \ref{LT} we see $P[\tau_s<1]\le e^{\lambda}E[e^{-\lambda\tau_s}]=e^{\lambda}e^{-cs\lambda^\beta}$, therefore $E[m_\tau[0,1]^p]\le p\int_0^\infty s^{p-1}P[\tau_s<1]ds\le e^{\lambda}p\int_0^\infty s^{p-1}e^{-cs\lambda^\beta}ds<\infty$.
\end{proof}

\section{Requantization through Resampling Identity}\label{ANO}

Note that in the context of Section \ref{LQG}, when we want to construct the quantized measure for a Radon measure $\sigma(dz)$ on a domain $D$, an almost necessary condition is the finiteness of energy, i.e. $\iint_{\bar D\times \bar D}\frac{\sigma(dx)\sigma(dy)}{|x-y|^{d-\varepsilon}}<\infty$ for some $d>0$ and any $\varepsilon\in(0,d)$. Unluckily, for random measures such as those measures on random fractals in this paper, the finiteness of their energy is usually hard to prove. However, in our case, as we have already constructed a quantum natural measure and its dequantized version, through the resampling identity which we will discuss in detail shortly we can avoid this problem.

We start with the notations we use in this section.

\begin{itemize}
\item  $D$ is a simply connected domain in $\mathbb{C}$;
\item  $\Gamma$ is a configuration sample of $\text{Obj}(d\Gamma)$ in $D$ (in this paper $\text{Obj}$ refers to the law of SLE or CLE;
\item  $h$ is a Gaussian field on a domain $D$ with correlation $K(x,y)=-\log|x-y|+g(x,y)$ where $g$ is continuous over $\bar D\times\bar D$;
\item  $\gamma\in\mathbb{R}$ is a constant fixed throughout this section;
\item  $\textbf{P}_D(dh)$ and $\textbf{P}_D^{(\gamma,z)}(dh)$ denote the laws of $h$ and $h+\gamma K(\cdot,z)$ respectively, viewed as probability measures on $H^{-1}(D)$;
\item  $\hat K(z)=\lim\limits_{\varepsilon\to 0}\log\varepsilon+\mathrm{Var}(\theta_z^\varepsilon,h)$ for $h\sim\textbf{P}_D(dh)$;
\item  $\nu_0(dz;D,\Gamma)$ is a Radon measure on $D$, depending on $\Gamma$;
\item $\mathcal{M}$ is the collection of signed measures $\rho$ on $D$ such that $\iint{|K(x,y)||\rho|(dx)|\rho|(dy)}<\infty$.
\end{itemize}

As mentioned above, we construct the quantized version of $\nu_0(dz;D,\Gamma)$ on $D$, as suring that we have already known the another measure, which will be the dequantized quantum natural measure in application that turns out to be the same on $D$ after averaging over $\text{Obj}(d\Gamma)$ and satisfies a resampling identity (see \eqref{SP}). Concretely, we suppose the following \textbf{uniqueness assumptions} hold throughout this section.

\begin{enumerate}[1)]
\item There is a measure $\mu_0(dz;D,\Gamma)$ satisfies the \textbf{resampling identity} 
\begin{equation*}
\textbf{P}_D^{(\gamma,z)}(dh)\mu_0(dz;D,\Gamma)\text{Obj}(d\Gamma)=e^{-\frac{1}{2}\gamma^2\hat K(z)}\mu_h(dz;D,\Gamma)\textbf{P}_D(dh)\text{Obj}(d\Gamma);
\end{equation*}
where $\mu_h(dz;D,\Gamma)$ is some measure on $D$ depending on $h$;
\item For each compact set $K\subset D$, $\mathbb{E}[\nu_0(K;D,\Gamma)]<\infty$;
\item The measure $\mathbb{E}[\nu_0(dz;D,\Gamma)]$ on $D$ is equal to $\mathbb{E}[\mu_0(dz;D,\Gamma)]$.
\end{enumerate}

Then we consider the joint measure $\textbf{P}_D^{(\gamma,z)}(dh)\nu_0(dz;D,\Gamma)\text{Obj}(d\Gamma)$, which is a finite measure on the Polish space $H^{-1}(D)\times K\times \text{Obj}$ (with a little abuse of notation, we keep $\rm{Obj}$ standing for the measure space where $\rm Obj(d\Gamma)$ is defined) for any compact set $K \subset D$. In view of this, we can apply the following measure disintegration theorem (see e.g. Theorem 10.6.6 in \cite{bogachev2007measure}).

\begin{theorem}
Let $X$ and $Y$ be two Polish spaces, and $\mu$ be a finite Borel measure on $Y$. Let $\pi: Y\to X$ be a Borel-measurable function, and $\nu$ be the pushforward measure $\pi_\star(\mu) =\mu\circ\pi^{-1}$. Then there exists a $\nu$-a.e. uniquely determined family of finite Borel measures $\{\mu_x\}_{x\in X}$  on $Y$ such that:
\begin{itemize}
\item  the function $x\mapsto \mu _{x}$ is Borel measurable;
\item  for $\nu$-a.e. $x\in X$, $\mu _{x}\left(Y\setminus \pi ^{-1}(x)\right)=0$;
\item  for every Borel measurable function $f : Y \to [0, \infty]$, we have
\begin{equation*}
\int _{Y}f(y)\,\mathrm {d} \mu (y)=\int _{X}\int _{\pi ^{-1}(x)}f(y)\,\mathrm {d} \mu _{x}(y)\mathrm {d} \nu (x).
\end{equation*}
In particular for any event $E \subset Y$, $\mu (E)=\int _{X}\mu _{x}\left(E\right)\,\mathrm {d} \nu (x)$.
\end{itemize}
\end{theorem}

According to the third assumption, the marginal law of $h$ in $\textbf{P}_D^{(\gamma,z)}(dh)\nu_0(dz;D,\Gamma)|_K\text{Obj}(d\Gamma)$ are absolutely continuous with respect to $\textbf{P}_D$. Also note that given $h$, the conditional marginal law $\text{Obj}_h(d\Gamma)$ of $\Gamma$ is absolutely continuous with respect to $\text{Obj}(d\Gamma)$. Indeed, for a null set $L$ for $\text{Obj}(d\Gamma)$, notice that the mass of $\textbf{P}_D^{(\gamma,z)}(dh)\nu_0(dz;D,\Gamma)\text{Obj}(d\Gamma)$ on $H^{-1}(D)\times K\times L$ is zero. Then we must have $\text{Obj}_h(L)=0$ for $\textbf{P}_D$ a.e. $h$. Therefore, we can write down the following resampling identity on $H^{-1}(D)\times K\times L$ for $\nu_0(dz;D,\Gamma)$:
\begin{equation}\label{SP}
\textbf{P}_D^{(\gamma,z)}(dh)\nu_0(dz;D,\Gamma)\text{Obj}(d\Gamma)=e^{-\frac{1}{2}\gamma^2\hat K(z)}\nu_h(dz;D,\Gamma)\textbf{P}_D(dh)\text{Obj}(d\Gamma)
\end{equation}
where $e^{-\frac{1}{2}\gamma^2\hat K(z)}\nu_h(dz;D,\Gamma)$ is the disintegration over $h$ and $\Gamma$ (we exchange the position of $\textbf{P}_D(dh)$ and $\text{Obj}(d\Gamma)$ since they are independent). Since $K$ is arbitrary,  the resampling identity holds on $H^{-1}(D)\times D \times L$. In the following we will see that it can be viewed as the quantized measure of $\nu_0(dz;D,\Gamma)$.

\begin{proposition}\label{SMV}
For $\text{Obj}(d\Gamma)$-a.e.\ $\Gamma$, the above $\nu_h(dz;D,\Gamma)$ is the quantized measure with respect to $\nu_0(dz;D,\Gamma)$ in the sense of Shamov's axiomatic construction of GMC \cite{shamov2016gaussian}. That is,
\begin{itemize}
\item  $\nu_h(dz;D,\Gamma)$ is measurable with respect to $h$;
\item  $e^{-\frac{1}{2}\gamma^2\hat K(z)}\mathbf{E}_D[\nu_h(dz;D,\Gamma)]=\nu_0(dz;D,\Gamma)$;
\item Almost surely for every fixed and deterministic Borel measurable function $\xi$ of the form $\xi(z)=K\rho(z)$ with $\rho\in\mathcal{M}$, $\nu_{h+\xi}(dz;D,\Gamma)(dz)=e^{\gamma\xi(z)}\nu_h(dz;D,\Gamma)$.
\end{itemize}
\end{proposition}
\begin{proof}
The first claim is very clear. The second claim easily follows from the resampling identity. We now focus on the third claim. It is well known that when $h$ is under the law of $\textbf{P}_D(dh)$, the Radon-Nikodym derivative of $h+\xi$ with respect to $h$ is $e^{\frac{1}{2}(\rho,K\rho)-(h,\rho)}$. Then since $\textbf{P}_D^{(\gamma,z)}(dh)$ is the law after adding a term $\gamma K(\cdot,z)$, the Radon-Nikodym derivative of $h+\xi$ with respect to $h$ is $e^{\frac{1}{2}(\rho,K\rho)+\gamma\xi(z)-(h,\rho)}$ for $h$ is a sample of $\textbf{P}_D^{(\gamma,z)}(dh)$. Therefore, after replacing $h$ by $h+\xi$ in the resampling identity, the third claim follows.
\end{proof}

In the next proposition, we point out that when the field $h$ can be decomposed to two independent fields, the measure $\nu_h(dz;D,\Gamma)$ has a property similar to \textbf{locality}.

\begin{proposition}\label{RST}
For a.e.\ configuration $\Gamma$, conditioned on $\Gamma$, suppose $K\subset D$ is a subdomain. If one can sample $\mathbf{P}_D(dh)$ by independently sampling $h|_K\sim\mathbf{P}_K(dh|_K)$ and $h|_{D\backslash K}\sim\mathbf{P}_{D\backslash K}(dh_{D\backslash K})$ then setting $h=h|_K+ h|_{D\backslash K}$, the restriction of $\nu_h(dz;D,\Gamma)$ on such $K$ is measurable with respect to the restriction of the field $h|_K$.
\end{proposition}
\begin{proof}
Note that by definition, to sample $\textbf{P}_D^{(\gamma,z)}(dh)$ we can also first sample $h|_K\sim\textbf{P}^{(\gamma,z)}_K(dh|_K)$ and $h|_{D\backslash K}\sim\textbf{P}^{(\gamma,z)}_{D\backslash K}(dh_{D\backslash K})$ and then set $h=h|_K+ h|_{D\backslash K}$. Therefore, the resampling identity (restricted in the subspace $H^{-1}(D)\times K\times \text{Obj}$) can be written as
\begin{equation*}
\begin{aligned}
\textbf{P}^{(\gamma,z)}_{D\backslash K}(dh_{D\backslash K})\textbf{P}_K^{(\gamma,z)}(dh|_K)&\nu_0(dz;D,\Gamma)|_K\text{Obj}(d\Gamma)\\=e^{-\frac{1}{2}\gamma^2 \hat K(z)}&\nu_h(dz;D,\Gamma)|_K\textbf{P}_{D\backslash K}(dh_{D\backslash K})\textbf{P}_K(dh|_K)\text{Obj}(d\Gamma).
\end{aligned}
\end{equation*}
However, since $z\in K$, $K(z,\cdot)$ is a continuous function in $D\backslash K$, thus the laws $\textbf{P}^{(\gamma,z)}_{D\backslash K}(dh_{D\backslash K})$ and $\textbf{P}_{D\backslash K}(dh_{D\backslash K})$ are mutually absolutely continuous; we denote $F(z)$ for the Radon-Nikodym derivative between them. Then we can write
\begin{equation*}
\begin{aligned}
F(z)\textbf{P}_K^{(\gamma,z)}(dh|_K)&\nu_0(dz;D,\Gamma)|_K\textbf{P}_{D\backslash K}(dh_{D\backslash K})\text{Obj}(d\Gamma)\\=e^{-\frac{1}{2}\gamma^2 \hat K(z)}&\nu_h(dz;D,\Gamma)|_K\textbf{P}_{D\backslash K}(dh_{D\backslash K})\textbf{P}_K(dh|_K)\text{Obj}(d\Gamma).
\end{aligned}
\end{equation*}
Now for any $U\subset K$ and a smooth function $w$ supported in $D\backslash K$, we have
\begin{equation*}
\mathbb{E}\left[\int\textbf{1}_U(z)e^{(h|_{D\backslash K},w)}\nu_h(dz;D,\Gamma)\right]=\mathbb{E}\left[\nu_h(U;D,\Gamma)\right]\mathbb{E}\left[e^{(h|_{D\backslash K},w)}\right],
\end{equation*} which shows that $\nu_h(dz;D,\Gamma)|_K$ is independent from $h|_{D\backslash K}$.
\end{proof}

\begin{remark}\label{1D}
We can also follow the above procedure for measures supported on the boundary of a domain. More precisely, suppose $D$ is simply connected and $\lambda(dx;D,\Gamma)$ is a finite Radon measure supported on $\partial D$. Then we can propose assumptions similarly and construct the quantized measure $\lambda_h(dx;D,\Gamma)$ as above. In particular, Proposition \ref{SMV} and \ref{RST} also holds for $\lambda_h(dx;D,\Gamma)$ (to restate Proposition \ref{RST}, the subdomain $K$ should share a segment of boundary with $\partial D$).
\end{remark}

\section{SLE Boundary Touching Points}\label{SLE}

In this section we consider an ${\rm SLE}_\kappa(\rho)$ curve $\eta$ in $\mathbb{H}$ from $0$ to $\infty$ with a single force point located at $0+$, where $\kappa\in(0,4)$ and $\rho\in\left(-2,\frac{\kappa}{2}-2\right)$. Write $d=\frac{(\rho+4)(\kappa-4-2\rho)}{2\kappa}$ for the Hausdorff dimension of $\eta\cap\mathbb{R}_+$.

In Section \ref{CS} we first construct the quantum natural measure on SLE boundary touching points using the coupling of SLE and LQG, and after averaging over LQG we will construct the Euclidean natural measure. Then in Section \ref{RES} we show a resampling identity, which connects the quantum natural measure and Euclidean natural measure and enables us to construct the quantized version of those natural measures in Definition \ref{def2}. Finally in Section \ref{UNS}, we prove the uniqueness part of Theorem \ref{thm2} by characterizing quantized measures through stable subordinators. Our arguments bear a similar flavor to Lemma 5.39 in \cite{holden2021convergence}.

\subsection{Construction of the Natural Measure}\label{CS}

Suppose we are in $(\mathbb{H},0,\infty)$, and let $h$ be an independent quantum wedge of weight $(\rho+4)$ on it, with its circle average embedding. By Theorem 6.16 in \cite{duplantier2020liouville}, we know that the law of beaded surface on the right of $\eta$ has the law of a quantum wedge of weight $(\rho+2)$. Therefore we can construct the quantum natural measure of boundary intersecting points through its Poissonian structure of thin quantum wedges. We parametrize $\eta$ by its quantum length.

\begin{proposition}\label{SUB}
The $\left(1-\frac{2(\rho+2)}{\kappa}\right)$-occupation measure of $\{s\ge 0:\eta(s)\in\mathbb{R}_+\}$ on $[0,\infty)$ exists, which we denote by $m$. 
\end{proposition}
We call $\nu_h(\cdot,\eta)=\eta\circ m$ the \textbf{quantum boundary touching measure} of $\eta$. Note that $\nu_h(\cdot,\eta)$ is a finite Borel measure supported on $\eta\cap\mathbb{R}_+$. Although this proposition has already appeared in various literature, see e.g.\ Lemma 2.13 in \cite{SLEwelding} and Lemma 2.6 of \cite{GHM15}, we still give a proof here for the sake of completeness.
\begin{proof}
Since the thin quantum wedge is a Poisson point process of quantum disks, using the same method of the proof of Proposition 4.18 in \cite{duplantier2020liouville}, we can see that the law of the left length of the bubbles of a quantum wedge of weight $W$ coincides with the law of the lengths of the excursions from 0 of a Bessel process of dimension $4W/\gamma^2$. That is, $\{s\ge 0:\eta(s)\in\mathbb{R}_+\}=\{t\ge 0: Y_t=0\}$ where $Y$ is a Bessel process of dimension $4W/\gamma^2$. Note that the latter set is equal to the range of the right-continuous inverse of the local time at 0 of $Y$, which is a $\left(1-\frac{2W}{\gamma^2}\right)$ stable subordinator. Taking $W=\rho+2$ we get the result.
\end{proof}

In the following, we need the following change rule of $\nu_h$ under a translation in the Cameron-Martin space.

\begin{proposition}\label{CM}
For a.e. $h$, for a continuous function $\xi$, we have that $\nu_{h+\xi}(dz;\eta)=e^{(\gamma-\frac{2}{\gamma})\xi}\nu_h(dz;\eta)$.
\end{proposition}
\begin{proof}
For the case that $\xi$ is a constant, the result quickly follows from the scaling property of quantum length. For the general case, one can approximate $\xi$ by piecewise constant functions, see e.g. Lemma 6.19 in \cite{duplantier2020liouville}.
\end{proof}

We now take a Dirichlet boundary GFF $h_0$ in place of the quantum wedge $h$. The measure $\nu_{h_0}$ makes sense thanks to the absolute continuity between the law of Dirichlet GFF and quantum wedge.

\begin{definition}
Let $\hat K_0(z)=\lim\limits_{\varepsilon\to 0}\log\varepsilon+\mathrm{Var}(\tilde \theta_z^\varepsilon,h)$, where $\tilde \theta_z^\varepsilon$ is the uniform measure on the half circle $\partial B(z,\varepsilon)\cap\mathbb{H}$ with total mass being $1$. Define 
\begin{equation}\label{equa4.1}
\nu_0(\cdot;\eta)=e^{-\frac{1}{4}\gamma^2\left(1-\frac{2(\rho+2)}{\kappa}\right)^2 \hat K_0(x)}\mathbf{E}\nu_{h_0}(\cdot;\eta)
\end{equation}
and call it the dequantized boundary touching measure of $\eta$.
\end{definition}

\begin{remark}\label{EPT}
We remark that when one uses other variants of GFF as LQG background to define its quantum natural measure, averaging over it will give the same measure up to a constant. Recall that the quantum wedge $h$ can be decomposed to the independent sum of this Dirichlet boundary GFF $h_0$ and a random harmonic function $\mathfrak{h}$. Therefore, let $\hat K(z)=\lim\limits_{\varepsilon\to 0}\log\varepsilon+\mathrm{Var}(\tilde \theta_z^\varepsilon,h)$, by double expectation and Proposition \ref{CM}, we have
\begin{equation*}
\mathbf{E}\nu_h(dx)=\mathbf{E}\nu_{h_0+\mathfrak{h}}(dx)=\mathbf{E}e^{\frac{\gamma}{2}\left(1-\frac{2(\rho+2)}{\kappa}\right)\mathfrak{h}}\mathbf{E}\nu_{h_0}(dx),
\end{equation*}
then $e^{-\frac{1}{4}\gamma^2\left(1-\frac{2(\rho+2)}{\kappa}\right)^2\hat K(x)}\mathbf{E}\nu_h(dx)=Ce^{-\frac{1}{4}\gamma^2\left(1-\frac{2(\rho+2)}{\kappa}\right)^2 \hat K_0(x)}\mathbf{E}\nu_{h_0}(dx)$. Therefore averaging over $h_0$ and averaging over the quantum wedge $h$ only differ by a multiplicative constant $C=e^{\frac{1}{2}\gamma\left(1-\frac{2(\rho+2)}{\kappa}\right)\mathbf{E}(h,\tilde\theta_z^\varepsilon)}$. Since this average of $h$ only depends on its embedding into $\mathbb{H}$, the constant $C$ does not depend on $z$.
\end{remark}

We need to show that this expectation \eqref{equa4.1} in the above definition is finite.

\begin{proposition}\label{FI}
The above definition of $\nu_0$ is well-defined and have finite expectation on any compact set.
\end{proposition}
\begin{proof}
According to Remark \ref{EPT}, we can choose another LQG surface to define $\nu_0$ for convenience. In this proof we are in the context of \cite{SLEwelding} (see Section \ref{QW}, or see notations in Section 4.1 in \cite{SLEwelding}). By absolute continuity we choose a sample of $\mathcal{M}_2^{\text{disk}}(\rho+4)$ with left and right boundary length being $l,r$ as the LQG background, in place of the above quantum wedge. Then the ${\rm SLE}_\kappa(\rho)$ curve running on this quantum disk will be the conformal welding of independent samples of $\mathcal{M}_2^{\text{disk}}(2)$ and $\mathcal{M}_2^{\text{disk}}(\rho+2)$(see Proposition 4.1 in \cite{SLEwelding}). Thus the quantum length $L$ of ${\rm SLE}_\kappa(\rho)$ curve has a density proportional to $|\mathcal{M}_2^{\text{disk}}(2;l,x)||\mathcal{M}_2^{\text{disk}}(\rho+2;x,r)|$. Note that by scaling we have $|\mathcal{M}_2^{\text{disk}}(W;l,r)|\le C (l+r)^{-1-2W/{\gamma^2}}$. In particular, $L$ has a tail no heavier than $x^{-\frac{2(\rho+4)}{\gamma^2}-2}$.

In the notation of Definition 3.5 of \cite{SLEwelding}, the total mass of $\nu_h$ now equals $T$ conditioned on the left and right boundary lengths of a $(\rho+2)$-weight thin quantum disk being $L$ and $r$. By scaling property (see Proposition 3.6 in \cite{SLEwelding}), the conditional expectation of $T$ in a thin quantum disk of weight $W$ given the left and right boundary lengths $l$ and $r$ is no more than $C(l+r)^{1-2W/\gamma^2}$ for some constant $C$. In particular the expectation of $T$ will no more than
\begin{equation*}
C\int_0^\infty(x+r)^{1-\frac{2(\rho+2)}{\gamma^2}}(l+x)^{-1-\frac{4}{\gamma^2}}(x+r)^{-1-\frac{2(\rho+2)}{\gamma^2}}dx<\infty,
\end{equation*}
i.e. the total mass of $\nu_h$ has finite expectation. The factor $\hat K(z)$ is uniformly bounded over a compact set, thus the result follows.
\end{proof}

The above proposition also verifies the finite expectation condition in Definition \ref{def2} for $\nu_0$. We now prove that $\nu_0$ satisfies the conformal Markov property in Definition \ref{def2} through a direct calculation. 

\begin{proposition}
For any stopping time $t>0$ such that $\eta(t)\in\mathbb{R}_+$, conditioned on $\eta[0,t]$, the joint law of $(\phi_t(\eta), |(\phi^{-1}_t)' |^{-d} \nu_0\circ\phi_t^{-1})$ is equal to the original joint law of $(\eta,\nu_0)$, where $\phi_t:\mathbb{H}\backslash H_t\to \mathbb{H}$ is any conformal map such that $\phi_t(\eta_t)=0$.
\end{proposition}
\begin{proof}
Suppose $\phi_t$ is a map described in statement of the proposition. According to the conformal Markovian property and scaling property of SLE, it is easy to see that $\phi_t(\eta)$ has the same law as $\eta$. Let $h_t=h\circ\phi_t$ denote the Dirichlet GFF on $\mathbb{H}\backslash H_t$, then we have 
\begin{equation*}
\nu_{h_0}\circ\phi_t^{-1}=e^{\frac{\gamma}{2}\left(1-\frac{2(\rho+2)}{\kappa}\right)Q\log|(\phi_t^{-1})'|}\nu_{h_t}.
\end{equation*}
Therefore, if we denote $\textbf{E}_t$ for averaging over $h_t$, taking $y=\phi_t(x)$, $\delta|\phi_t'|=\varepsilon$ then we have
\begin{equation*}
\begin{aligned}
\mathcal{\nu}_0\circ\phi_t(dx)&=e^{-\frac{1}{4}\gamma^2\left(1-\frac{2(\rho+2)}{\kappa}\right)^2\hat K(\phi_t(x))}\textbf{E}\nu_{h_0}(\phi_t(dx))\\
&=e^{-\frac{1}{4}\gamma^2\left(1-\frac{2(\rho+2)}{\kappa}\right)^2\left(\lim\limits_{\varepsilon\to 0}\log\varepsilon+\mathrm{Var}(h_0,\theta_y^{\delta})\right)}|\phi_t'(x)|^{\frac{\gamma}{2}\left(1-\frac{2(\rho+2)}{\kappa}\right)Q}\textbf{E}_t\nu_{h_t}(dx)\\
&=e^{-\frac{1}{4}\gamma^2\left(1-\frac{2(\rho+2)}{\kappa}\right)^2\left(\lim\limits_{\varepsilon\to 0}\log\varepsilon+\mathrm{Var}(h_t\circ\phi_t^{-1},\theta_y^{\delta})\right)}|\phi_t'(x)|^{\frac{\gamma}{2}\left(1-\frac{2(\rho+2)}{\kappa}\right)Q}\textbf{E}_t\nu_{h_t}(dx)\\
&=|\phi_t'(x)|^de^{-\frac{1}{4}\gamma^2\left(1-\frac{2(\rho+2)}{\kappa}\right)^2\left(\lim\limits_{\varepsilon\to 0}\log\delta+\mathrm{Var}(h_t',\theta_x^{\delta})\right)}\textbf{E}_t\nu_{h_t}(dx)\\
&=|\phi_t'(x)|^d\nu_0(dx)
\end{aligned}
\end{equation*}
where we use the identity
\begin{equation}\label{KPZE}
\frac{1}{4}\gamma^2\left(1-\frac{2(\rho+2)}{\kappa}\right)^2-\frac{\gamma}{2}\left(1-\frac{2(\rho+2)}{\kappa}\right)Q+d=0
\end{equation}
in the third line.  The validity of changing of coordinates in the above calculation is explained in the proof of Proposition 2.19 in \cite{benoist2017natural}. The scaling property of $\nu_0$ also comes from similar calculation (by replacing the above $\phi_t$ with $\phi:z\mapsto rz$). Therefore $\nu_0(dz;\eta)$ is a natural measure on SLE boundary touching points in the sense of Definition \ref{def2}.
\end{proof}

\begin{remark}\label{KPZ}
The above identity \eqref{KPZE} corresponds to a version of boundary KPZ relation. In the bulk case, the KPZ relation can be written as $\textbf{Q}(a,d)=\textbf{Q}(\gamma,2)$ where $\textbf{Q}(a,D)\triangleq\frac{a}{2}+\frac{D}{a}$ for fractals with $D$ and $a$ being its Euclidean and quantum dimension respectively. We will see this KPZ relation several times in the paper, see e.g. Proposition \ref{CFM}.
\end{remark}

\begin{remark}
In \cite{zhan2021}, it is proved that the Minkowski content of $\eta\cap \mathbb{R}_+$ exists and has finite $d$-dimensional energy for any $\varepsilon>0$. Therefore as discussed in Section \ref{sec:discussion}, we can also use the Minkowski content to define our boundary touching measure, and it is easy to see the Minkowski content satisfies those axioms in Definition \ref{def2}. Indeed, according to the uniqueness part of Theorem \ref{thm2}, up to a multiplicative constant, it is equal to our natural measure constructed above.
\end{remark}

\subsection{Resampling Identity}\label{RES}

In this subsection we will show the resampling identity for $\nu_0(dz;\eta)$ and $\nu_h(dz;\eta)$ constructed in the above subsection.

\begin{proposition}\label{IED}
Denote ${\rm SLE}(d\eta)$ for the law of the ${\rm SLE}_\kappa(\rho)$ curve above on the half-plane. Then, for $\nu_0$ and $\nu$ in the above subsection, we have the resampling identity
\begin{equation*}
e^{-\frac{1}{4}\gamma^2\left(1-\frac{2(\rho+2)}{\kappa}\right)^2\hat K(z)}\nu_h(dz;\eta)\mathbf{QW}_{\mathbb{H}}(dh)\mathrm{SLE}(d\eta)=\mathbf{QW}_{\mathbb{H}}^{\left(\frac{\gamma}{2}\left(1-\frac{2(\rho+2)}{\kappa}\right),z\right)}(dh)\nu_0(dz;\eta)\mathrm{SLE}(d\eta).
\end{equation*}
\end{proposition}
\begin{proof}
This follows by direct calculation. For any bounded set $U$ and smooth function $\rho$, let $w=G^{W}\rho$ (here $G^W$ denotes the Green function of the quantum wedge). Therefore we have
\begin{equation*}
\begin{aligned}
&\mathbb{E}\left[\int e^{-\frac{1}{4}\gamma^2\left(1-\frac{2(\rho+2)}{\kappa}\right)^2\hat K(z)}\textbf{1}_U(z)e^{(h,\rho)}\nu_h(dz;\eta)g(\eta)\right]
\\=\ &\mathbb{E}\left[e^{(h,\rho)}\right]\mathbb{E}\left[\int e^{-\frac{1}{4}\gamma^2\left(1-\frac{2(\rho+2)}{\kappa}\right)^2\hat K(z)}\textbf{1}_U(z)\nu_{h+w}(dz;\eta)g(\eta)\right]
\\=\ &\mathbb{E}\left[e^{(h,\rho)}\right]\mathbb{E}\left[\int e^{-\frac{1}{4}\gamma^2\left(1-\frac{2(\rho+2)}{\kappa}\right)^2\hat K(z)}\textbf{1}_U(z)e^{\frac{\gamma}{2}\left(1-\frac{2(\rho+2)}{\kappa}\right)w(z)}\nu_{h}(dz;\eta)g(\eta)\right]
\\=\ &\mathbb{E}\left[e^{(h,\rho)}\right]\mathbb{E}\left[\int e^{-\frac{1}{4}\gamma^2\left(1-\frac{2(\rho+2)}{\kappa}\right)^2\hat K(z)}\textbf{1}_U(z)e^{\frac{1}{4}\gamma^2\left(1-\frac{2(\rho+2)}{\kappa}\right)^2\hat K(z)}e^{\frac{\gamma}{2}\left(1-\frac{2(\rho+2)}{\kappa}\right)w(z)}\nu_0(dz;\eta)g(\eta)\right]
\\=\ &\mathbb{E}\left[\int \textbf{1}_U(z)e^{(h+\frac{\gamma}{2}\left(1-\frac{2(\rho+2)}{\kappa}\right)G^W,\rho)}\nu_0(dz;\eta)g(\eta)\right]
\\=\ &\int\textbf{1}_U(z)g(\eta)e^{(h,\rho)}\textbf{QW}_{\mathbb{H}}^{\left(\frac{\gamma}{2}\left(1-\frac{2(\rho+2)}{\kappa}\right),z\right)}(dh)\nu_0(dz){\rm SLE}(d\eta).
\end{aligned}
\end{equation*}
This concludes the proof.
\end{proof}

Then we also need to check the uniqueness assumption in Section \ref{ANO} that the averages of natural measures in Definition \ref{def2} over SLE are the same. It is quite easy in the boundary touching measure case.

\begin{proposition}
Define a measure $m$ on $\mathbb{R}_+$ such that $m[0,r]=|r|^d$ for every $r>0$. For each natural measure $\mu_0$ in Definition \ref{def2}, $\mathbb{E}[\mu_0(dz)]$ is equal to $cm$ for a deterministic constant $c>0$.
\end{proposition}
\begin{proof}
According to the scaling property in Definition \ref{def2}, we have $\mathbb{E}[\mu_0(dz)]=|\phi'|^d\phi^{-1}\circ\mathbb{E}[\mu_0(dz)]$ for $\phi:z\mapsto |r|z$. Therefore $\mathbb{E}[\mu_0([0,r])]=|r|^d\mathbb{E}[\mu_0([0,1])]$.
\end{proof}

According to Section \ref{ANO} we can get the quantized version for each natural measure $\nu_0$.

\subsection{Uniqueness of the Quantum Natural Measure}\label{UNS}

\begin{proof}[Proof of Theorem \ref{thm2}, Uniqueness Part]
We have checked that our natural measure in Definition \ref{def2} satisfies the uniqueness assumptions in Section \ref{ANO}. Then introduce an independent LQG field $h$ s.t. $(\mathbb{H},0,\infty,h)$ has the law of a quantum wedge of weight $(\rho+4)$, and parametrize $\eta$ by its quantum length. According to the third property of Definition \ref{def2} and Remark \ref{1D}, we can get the quantized version $\mu=\mu(dz;\eta)$ of any natural measure $\mu_0(dz;\eta)$. 

Denote $\nu(dz;\eta)$ for the quantum boundary touching measure constructed in the above subsection. For $t>0$, let $t'=\inf\{s\ge 0:\nu(\eta[0,s]\cap \mathbb{R}_+)=t\}$, and let $X_t$ be the total $\mu$-mass of $\eta[0,t']\cap\mathbb{R}_+$. Denote $D_0$(resp. $D_l$) for the beaded (resp. unbeaded) region bounded by $\eta$ and $\mathbb{R}$, and let $D_t$ be the closure of the unbounded component of $\overline D_0\backslash\{\eta(t')\}$. According to Theorem \ref{BT}, we see that $D_t$ is a thin quantum wedge of weight $(\rho+2)$ for each $t$, and $D_l$ is a quantum half-plane. By the symmetry of quantum half-plane, we know that $\mathcal{W}_l^t=(D_l,h,\eta(t'),\infty)$ has the law of a quantum half-plane as well.

\begin{figure}[!htb]
\centering
\includegraphics[width=0.5\textwidth]{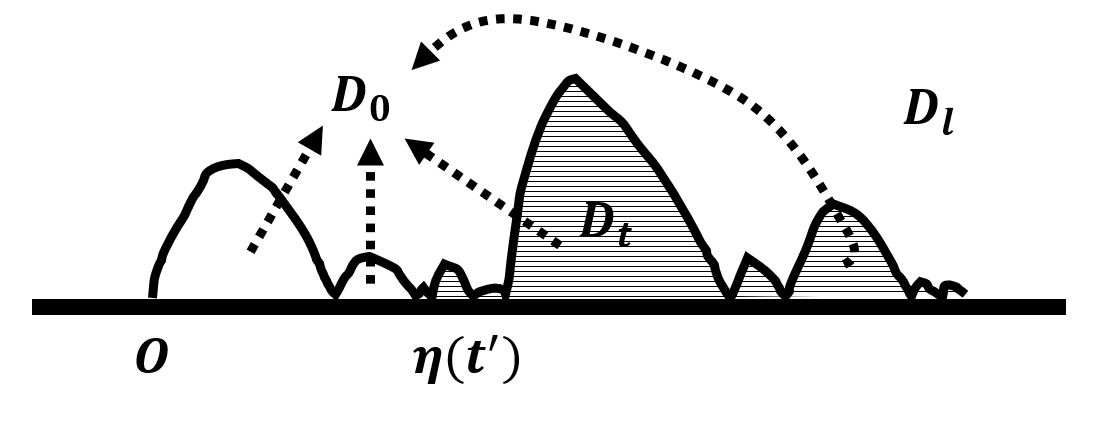}
\caption{An illustration of the proof. The quantum surface on the left and right of $\eta$ is $D_l$ and $D_0$ respectively, and the shaded region corresponds to $D_t$.}
\label{Fig.main2}
\end{figure}

Let $\psi_t$ be the conformal transformation from $\mathbb{H}$ to $D_l\cup D_t$, mapping $0$ to $\eta_l(t')$ such that $h_t=h\circ\psi_t+Q\log|\psi_t'|$ is in its circle average embedding. Since we have the resampling identity(see Proposition \ref{IED}), we only need to check its right hand side is invariant under $h\mapsto h_t$ in order to prove $\nu_{h_t}(\psi_t(dz))=\nu_{h}(dz)$. Indeed, This comes from the following calculation
\begin{equation*}
\begin{aligned}
\textbf{QW}^{(\alpha,\psi_t(z))}(dh_t)&=\lim\limits_{\varepsilon\to 0}e^{-\frac{1}{2}\alpha^2\hat K_t(\psi_t(z))}\varepsilon^{\frac{\alpha^2}{2}}e^{\alpha(h_t)_\varepsilon(\psi(z))}\textbf{QW}(dh_t)\\
&=e^{-\frac{1}{2}\alpha^2\hat K_t(\psi_t(z))}\lim\limits_{\delta\to 0}(\delta|\psi_t'(z)|)^{\frac{1}{2}\alpha^2}|\psi_t'(z)|^{-Q\alpha}e^{\alpha h_\varepsilon(z)}\textbf{QW}(dh)\\
&=e^{-\frac{1}{2}\alpha^2\hat K_t(\psi_t(z))}|\psi_t'(z)|^{-d}\textbf{QW}^{(\alpha,z)}(dh)
\end{aligned}
\end{equation*}
for any $\alpha\in\mathbb{R}$ and the conformal Markov property of $\nu_0$. By Proposition \ref{RST} we can see $\nu_{h_t}$ only depends on the field restricted in $D_t$. Define $\mathcal{W}_t$ for the quantum wedge $(D_t,h,\eta(t'),\infty)$. One thing worth of mentioning is that in the above calculation, the value of $\psi_t'$ on the boundary is well-defined thanks to Schwarz reflection principle.

Therefore the process $(X_{s+t})_{s\ge 0}$ is determined by $(\mathcal{W}_l^t,\mathcal{W}_t)$ in the same way as $(X_t)_{t\ge 0}$ is determined by $(\mathcal{W}_l^0,\mathcal{W}_0)$, and thus $(X_t)$ has independent and stationary increments. By adding-constant invariance of quantum wedges, the distribution of $X_t/t$ does not depend on $t$. Thus $(X_t)$ is a trivial stable subordinator (see Corollary \ref{TVI}), such that $X_t=t\textbf{E}X_1$ for every $t$ almost surely. The uniqueness then follows.
\end{proof}

\section{CLE Pivotal Points}\label{PVT}

In this section we prove Theorem \ref{thm3}. Throughout this section $\kappa'\in(4,8)$, and $d=2-\frac{(12-\kappa')(4+\kappa')}{8\kappa'}$ is the Hausdorff dimension of ${\rm CLE}_{\kappa'}$ pivotal points.

\subsection{The Whole-Plane CLE and Explorations}\label{BC}

We first construct whole-plane ${\rm CLE}_{\kappa'}$ using branching whole-plane ${\rm SLE}_{\kappa'}(\kappa'-6)$ process (see e.g. Section 3.2 in \cite{wholeCLE} for further reference). Suppose $\{\tilde\eta_z\}_{z\in\mathbb{Q}^2}$ is a branching whole-plane ${\rm SLE}_{\kappa'}(\kappa'-6)$ process starting from $\infty$, with its Loewner driving pair $(W^z,O^z):\mathbb{R}\to\partial \mathbb{D}\times\partial \mathbb{D}$. Let $\theta_z$ be the continuous version of $(\arg W^z-\arg O^z)$. Consider the collection $\mathcal{T}_z$ of all $t\in\mathbb{R}$ such that $\theta_z(t)\in 2\pi\mathbb{Z}$ and the last time $s<t$ with $\theta_z(s)\in 2\pi\mathbb{Z}$ satisfying $\theta_z(s)\neq\theta_z(t)$. Then $\mathcal{T}_z$ can be written as $\{t_{z,i}\}_{i\in\mathbb{Z}}$, and we let $\{\tau_{z,j}\}_{j\in\mathbb{Z}}$ be those $t_{z,i}$'s with $\theta_z(t_{z,i})-\theta_z(t_{z,i-1})=2\pi$. Denote $\sigma_{z,j}$ for the last time $s<\tau_{z,j}$ such that $\theta_z(\sigma_{z,j})\in2\pi\mathbb{Z}$. Concatenating the curve $\tilde\eta_z[\sigma_{z,j},\tau_{z,j}]$ together with the branch of $\{\tilde\eta_z\}_{z\in\mathbb{Q}^2}$ from $\tilde\eta_z(\tau_{z,j})$ to $\tilde\eta_z(\sigma_{z,j})$ we get a loop $\gamma_{z,j}$ (note that the latter segment has the law of a chordal ${\rm SLE}_{\kappa'}$). The whole-plane CLE configuration $\Gamma$ is defined as $\{\gamma_{z,j}:z\in\mathbb{Q},j\in\mathbb{Z}\}$.

\begin{figure}[H]
\centering  
\subfigure{
\label{Fig.sub.1}
\includegraphics[width=0.45\textwidth]{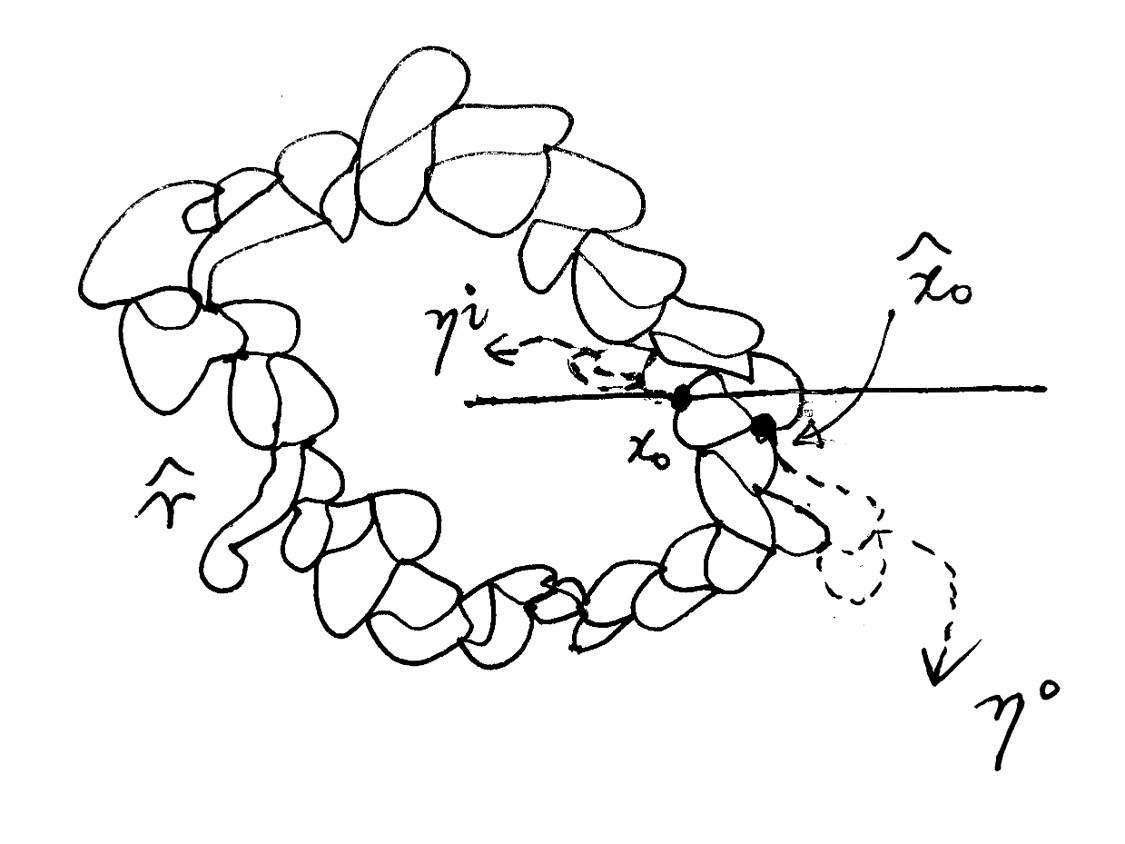}}
\subfigure{
\label{Fig.sub.2}
\includegraphics[width=0.25\textwidth]{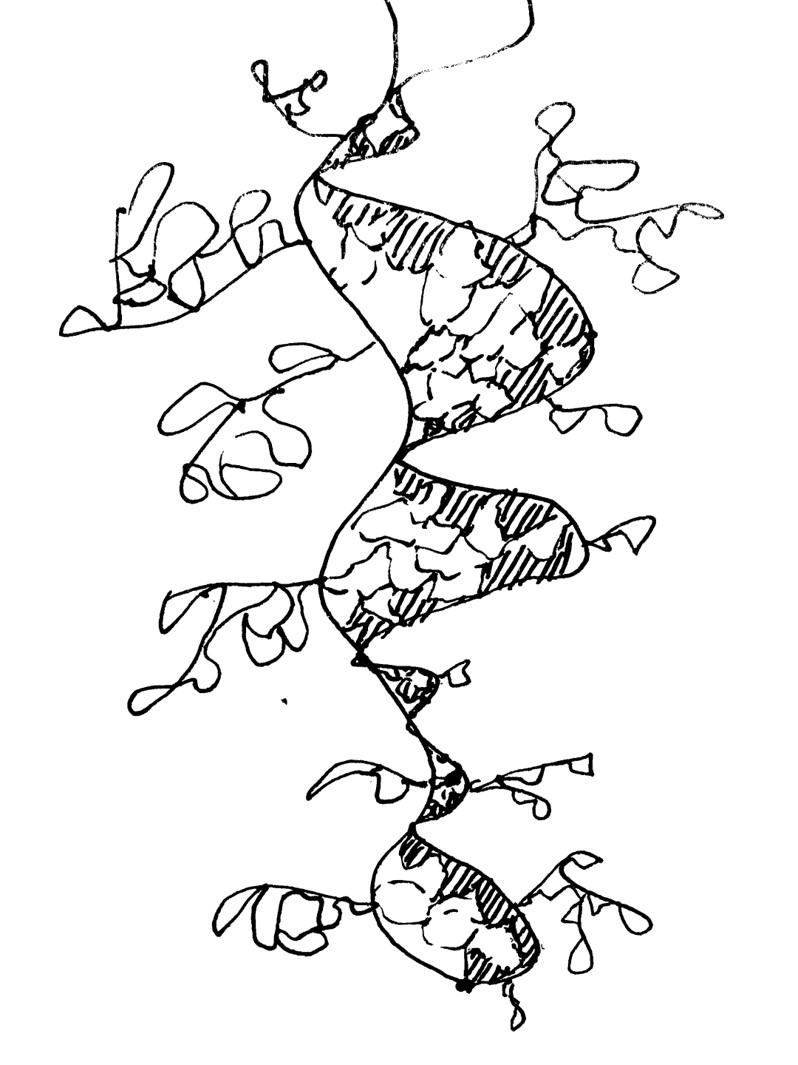}}
\caption{\textbf{Left}: Illustration of the pseudo-interface of a whole-plane CLE. \textbf{Right}: Illustration of the forested wedge $\tilde{\mathcal{W}}_0$ and the quantum wedge $\mathbb{W}_0$ (the shaded region). The end points of $\mathbb{W}_0$ are macroscopic pivotal points on the pseudo-interface $\eta_0$.}
\label{Fig.main}
\end{figure}

In reverse, given a whole-plane CLE configuration $\Gamma$ we are not able to recover the interface $\tilde\eta_0$; however, we can define its pseudo-interface as following. For each loop $\gamma\in\Gamma$, denote $\Upsilon(\gamma)$ and $\hat\Upsilon(\gamma)$ for the connected component of $\mathbb{C}\backslash \gamma$ which contains 0 and $\infty$ respectively. Let $\gamma^0$ be the innermost CLE loop in $\Gamma$ such that $\Upsilon(\gamma_0)$ contains $B(0,1)$, and let $x_0=\inf (\gamma^0\cap\mathbb{R}_+)$ which is on the boundary of $\Upsilon(\gamma_0)$. According to the Markov property of CLE (see Lemma 2.9 in \cite{wholeCLE}), conditioned on $\gamma^0$ and the loop configuration outside $\Upsilon(\gamma_0)$, the loop configuration in $\Upsilon(\gamma_0)$ has the law of an independent ${\rm CLE}_{\kappa'}$. Therefore one can define a counterclockwise radial exploration curve $\eta^i$ from $x_0$ to 0, which has the law of ${\rm SLE}_{\kappa'}(\kappa'-6)$. Let $\hat x_0$ be the end point of the segment $\hat\gamma$ of $\gamma_0$ from $x_0$ going clockwise until the first time it separates 0 from $\infty$. Similarly, thanks to the inversion invariance of CLE, conditioned on $\gamma^0$ and loop configuration outside $\hat\Upsilon(\gamma_0)$, the loop configuration in $\hat\Upsilon(\gamma_0)$ also has the law of an independent ${\rm CLE}_{\kappa'}$. Therefore one can also define a clockwise radial exploration curve $\eta^o$ from $\hat x_0$ to $\infty$. Let $\eta_0$ be the curve concatenating the reversal of $\eta^i$, $\hat\gamma$ and $\eta^o$ such that $\eta_0$ is a clockwise curve from 0 to $\infty$. Then $\eta_0$ has the law of a whole-plane ${\rm SLE}_{\kappa'}(\kappa'-6)$, and we call $\eta_0$ the \textbf{pseudo-interface} of $\Gamma$. Here we use the prefix \textit{pseudo} because $\eta_0$ is NOT the curve $\tilde\eta_0$ in the branching whole-plane ${\rm SLE}_{\kappa'}(\kappa'-6)$ process which generates $\Gamma$. Indeed if one attempts to regenerate $\Gamma$ with this pseudo-interface, the result will be a loop configuration which is not exactly $\Gamma$ but $\Gamma$ with some pivotal points of $\eta_0$ flipped.

Now introduce an independent LQG background $h$ such that $(\mathbb{C},h,0,\infty)$ is a quantum cone of weight $\left(4-\gamma^2\right)$ with circle average embedding. According to Theorem 1.17 in \cite{duplantier2020liouville}, since $\eta_0$ is a whole-plane ${\rm SLE}_{\kappa'}(\kappa'-6)$ process from $0$ to $\infty$, those components of $\mathbb{C}\backslash\eta_0$ which are not surrounded by its left or right boundary forms a quantum wedge $\mathcal{W}_0$ of weight $\left(2-\frac{\gamma^2}{2}\right)$, and $\mathcal{W}_0$ is decorated by a Poisson point process of loop-trees to become a forested wedge $\tilde{\mathcal{W}}_0$. In each bead of $\mathcal{W}_0$, we do the chordal exploration between its two ends, and the interface is a chordal ${\rm SLE}_{\kappa'}$ curve. Then those components connected to the right boundary of $\mathcal{W}_0$ become quantum disks, and concatenate them together forms a wedge $\mathbb{W}_0$ of weight $\left(\gamma^2-2\right)$. We call the marked points $P_0$ of $\mathbb{W}_0$ the \textbf{macroscopic pivotal points on the pseudo-interface} $\eta_0$. If 
\begin{equation}
\mathbb{E}[\nu_0(P_0\cap K,\mathbb{C},\Gamma)]<\infty\quad \mbox{ for any compact set $K$},
\end{equation}
we say that $\nu_0$ satisfies the \textbf{finiteness of expectation} property, which corresponds to the third condition of Definition \ref{def3}. 

\subsection{Construction of the Pivotal Measure}\label{CPP}

In this subsection we define a quantum pivotal measure on the configuration $\Gamma$. In parallel to Proposition \ref{SUB}, after parametrizing the right boundary of $\mathbb{W}_0$ by its quantum length, we can define a $\left(\frac{\kappa'}{4}-1\right)$-occupation measure on the pre-image of marked points. Pushing it forward we then define a measure on marked points of $\mathbb{W}_0$. For other points $z\in\mathbb{Q}^2$ we can also define this measure on marked points of $\mathbb{W}_z$ in parallel. The collection of those measures are defined as the quantum natural measure on CLE pivotal points, and we denote it by $\nu_h(dz;\mathbb{C},\Gamma)$. It is worth noting that this $\nu_h(\cdot;\mathbb{C},\Gamma)$ is not locally finite and assigns infinite mass to any open set intersecting $\Gamma$. 

\begin{remark}\label{LOC}
In the context of Proposition \ref{SUB}, if we parametrize $\eta_r$ by its quantum length, we can define a measure $m'$ on the cut point set $\mathcal{C}$. Since $\eta_r^{-1}(\mathcal{C})$ also has the law of the lengths of the excursions from 0 of a Bessel process of the same dimension, $m'$ has the same law as $m$. Furthermore, in the same spirit as the proof of uniqueness part of Section \ref{SLE}, we can show that $m=cm'$ for some deterministic constant $c\in(0,\infty)$. Combining the above two we have $m=m'$. This observation is vital to the locality property of the above natural measure on CLE pivotal points.
\end{remark}

For the case of a general simply connected domain $D\subset\mathbb{C}$, suppose there is a ${\rm CLE}_{\kappa'}$ configuration $\Gamma_D$ in $D$. Consider a whole-plane ${\rm CLE}_{\kappa'}$ configuration $\Gamma$ and choose a loop $\gamma_0$ surrounding $0$ in an arbitrary but fixed way, and denote the bounded component of $\mathbb{C}\backslash\gamma_0$ containing $0$ for $\Upsilon_0$. Define $\Gamma'$ as the result of changing loops in $\Upsilon_0$ to $\phi_D(\Gamma_D)$, where $\phi_D:D\to\Upsilon_0$ is some conformal map. Therefore the pullback of the restriction $\nu_h(dz;\mathbb{C},\Gamma')|_{\Upsilon_0}$ by $\phi_D$ defines a quantum natural measure on CLE pivotal points in the domain $D$. It is easy to check that this measure only depends on $\Gamma$ and $h'=h\circ\phi_D+Q\log|\phi_D'|$, therefore we can safely denote it as $\nu_{h'}(dz;D,\Gamma)$.

To get dequantized version $\nu_0(\cdot;D,\Gamma)$ of $\nu_h$, we will take a Dirichlet GFF $h_0$ on $D$ as LQG background as in Section \ref{SLE}.

\begin{definition}
For any simply connected domain $D\neq\mathbb{C}$, define
\begin{equation*}
\nu_0(\cdot;\Gamma,D)=e^{-\frac{1}{4}\gamma^2\left(1-\frac{2(\rho+2)}{\kappa}\right)^2 \hat K_0(x)}\textbf{E}\nu_{h_0}(\cdot\;\Gamma,D).
\end{equation*}
For the case of whole-plane, let $h_0^R$ be the Dirichlet GFF on $R\mathbb{D}$ for any $R>0$, and define 
\begin{equation*}
\nu_0(\cdot;\Gamma,\mathbb{C})|_{R\mathbb{D}}=e^{-\frac{1}{4}\gamma^2\left(1-\frac{2(\rho+2)}{\kappa}\right)^2 \hat K_0^R(x)}\textbf{E}\nu_{h_0^R}(\cdot\;\Gamma,\mathbb{C})|_{R\mathbb{D}}
\end{equation*}
(it is easy to check that this definition is consistent with different choices of $R$'s).
\end{definition}

We now check that this $\nu_0(\cdot;D,\Gamma)$ satisfies another two properties in Definition \ref{def3}, so it is a natural measure on CLE pivotal points. We first prove the finite expectation property on the macroscopic pivotal points on pseudo-interface.

\begin{proposition}
$\mathbb{E}[\nu_0(P_0\cap K;\mathbb{C},\Gamma)]<\infty$ for any compact set $K$.
\end{proposition}
\begin{proof}
We choose a quantum disk of weight $(4-\gamma^2)$ conditioned on its left and right boundary both being $l$ (therefore after welding it we would get a quantum sphere). Note that the left and right boundaries of $\mathbb{W}_0$ divide this quantum disk into three quantum disks of weight $W_1$, $W_2=2-\gamma^2/2$ and $W_3$ where $W_1+W_3=2-\gamma^2/2$. The rest of argument is same as Proposition \ref{FI}.
\end{proof}

\begin{proposition}\label{LOCAL}
The locality property in Theorem \ref{thm3} holds for $\nu_0$. That is, for given $D$ and configuration $\Gamma$, $\nu_0(dx;D,\Gamma)|_U$ is determined by the local geometry $\{l\cap U: l\in\Gamma\}$ for any subdomain $U\subset D$.
\end{proposition}
It is worth mentioning that in our position now, we can only conclude that \textbf{given} the domain $D$, for any subdomain $U\subset D$, $\nu_0(U;D,\Gamma)$ only depends on the loop configuration in $U$. However, to show the locality property, we need to show that we can tell the measure once we know $U$ and the loop configuration in it (even if we do not know the domain $D$).
\begin{proof}[Proof of Proposition \ref{LOCAL}]
We only need to check the case $D=\mathbb{C}$, and focus on the subdomain $U\subset R\mathbb{D}$ and sample a zero boundary GFF $h_0^U$ in $U$. Since we have the decomposition $h_0^R=h_0^U+\mathfrak{h}^U$ in $U$, by double expectation again (recall the calculation in Remark \ref{EPT}) it follows that
\begin{equation*}
\nu_0|_U(dx)=e^{-\frac{1}{8}\gamma^2\left(\frac{\kappa'}{4}-1\right)^2 \hat K_0^R(x)}\textbf{E}e^{\frac{1}{2}\gamma\left(\frac{\kappa'}{4}-1\right)\mathfrak{h}^U}\textbf{E}\nu_{h_0^U}(dx)=e^{-\frac{1}{8}\gamma^2\left(\frac{\kappa'}{4}-1\right)^2 \hat K_0^U(x)}\textbf{E}\nu_{h_0^U}(dx).
\end{equation*}
Also note that SLE quantum lengths can be realized through its quantum Minkowski content, which is local. Therefore, given $U$ and two intersecting SLE segments in $U$, one can just sample a zero boundary GFF in $U$ to parametrize one of them and obtain the quantum natural measure in $U$. Then by averaging over this zero boundary GFF one obtains $\nu_0$ restricted in $U$. Thanks to Remark \ref{LOC}, when there are two SLE segments intersecting each other, the outputs of taking quantum natural parametrization and then pushing forward the occupation measure of these two segments are the same. This is the locality we needed.
\end{proof}

\begin{remark}
The factor $e^{-\frac{1}{2}\tilde\gamma^2 \hat K_0(x)}$ is important when we get dequantized measure from quantum natural measure. It actually \textbf{removes} the influence of the geometric position of $x$ as illustrated in the proof above.
\end{remark}

\begin{proposition}\label{CFM}
The conformal coordinate change formula in Theorem \ref{thm3} holds for $\nu_0$. That is, For a conformal transformation $\psi:D\to D'$, the Radon-Nikodym derivative between $\nu_0(\psi(\cdot);D',\psi(\Gamma))$ and $\nu_0(\cdot;D,\Gamma)$ is
\begin{equation*}
\frac{\mathrm{d}\nu_0(\psi(\cdot);D',\psi(\Gamma))}{\mathrm{d}\nu_0(\cdot;D,\Gamma)}=|\psi'(x)|^d,
\end{equation*}
 where $d=2-\frac{(12-\kappa')(4+\kappa')}{8\kappa'}$ is the Hausdorff dimension of ${\rm CLE}_{\kappa'}$ pivotal points.
\end{proposition}
\begin{proof}
In parallel to Section \ref{CS}, we first do some calculation. Let $h_\psi=h_0\circ\psi^{-1}$ be a Dirichlet GFF on $D'$. Denote $\textbf{E}_\psi$ for the average under $h_\psi$. Then we have
\begin{equation*}
\begin{aligned}
\nu_0(\psi(dx);D',\psi(\Gamma))&=e^{-\frac{1}{8}\gamma^2\left(\frac{\kappa'}{4}-1\right)^2\hat K_0(\psi(x))}\textbf{E}_\psi\nu(\psi(dx);D',\psi(\Gamma))\\
&=e^{-\frac{1}{8}\gamma^2\left(\frac{\kappa'}{4}-1\right)^2\left(\lim\limits_{\varepsilon\to 0}\log\varepsilon+\mathrm{Var}(h_\psi\circ\psi^{-1},\theta_y^{\delta})\right)}\textbf{E}_\psi\nu(\psi(dx);D',\psi(\Gamma))\\
&=e^{-\frac{1}{8}\gamma^2\left(\frac{\kappa'}{4}-1\right)^2\left(\lim\limits_{\varepsilon\to 0}\log\varepsilon+\mathrm{Var}(h_\psi\circ\psi^{-1},\theta_y^{\delta})\right)}|\psi'(x)|^{\frac{1}{2}\gamma\left(\frac{\kappa'}{4}-1\right)Q}\textbf{E}\nu(dx;D,\Gamma)\\
&=|\psi'(x)|^de^{-\frac{1}{8}\gamma^2\left(\frac{\kappa'}{4}-1\right)^2\left(\lim\limits_{\varepsilon\to 0}\log\delta+\mathrm{Var}(h_\psi',\theta_x^{\delta})\right)}\textbf{E}\nu(dx;D,\Gamma)\\
&=|\psi'(x)|^d\nu_0(dx;D,\Gamma)
\end{aligned}
\end{equation*}
where $\delta|\psi'|=\varepsilon$, $y=\psi(x)$ and we note the KPZ relation $\frac{1}{8}\left(\gamma\left(\frac{\kappa'}{4}-1\right)\right)^2-\frac{1}{2}\left(\gamma\left(\frac{\kappa'}{4}-1\right)\right)Q+d=0$ as explained in Remark \ref{KPZ}.
\end{proof}

\begin{remark}
In Definition 5.18 of \cite{holden2021convergence}, one can also use the SLE boundary touching measure of its interface to construct the quantum natural measure $\nu_h(dz;D,\Gamma)$. The procedure in Section \ref{BC} could be seen as the radial version of the construction of the (quantum) pivotal measure in \cite{holden2021convergence}. The reason why we choose this radial version is that it can be extended to the whole-plane case. 
\end{remark}

\begin{remark}
Note that here we use the untruncated pivotal measure to keep the conformal invariance. It is different from those in the literature where one restricts it on those pivotal sites with four alternating arm to a Euclidean length more than $\varepsilon$ (see Section 4 in \cite{garban2014pivotal}) or on those produced by loops with quantum areas larger than $\varepsilon$ (see Section 5 of \cite{holden2021convergence}). As a price, dealing with this untruncated measure makes our statement of finiteness of expectation less straight forward.
\end{remark}

\subsection{Uniqueness}\label{U}

In parallel to the above case, we need to build the following resampling identity first. Its proof is by direct calculation, similar to Proposition \ref{IED}.
\begin{proposition}
For the natural measure $\nu_0$ and the quantum natural measure $\nu$ constructed in the above subsection, we have
\begin{equation*}
e^{-\frac{1}{8}\gamma^2\left(\frac{\kappa'}{4}-1\right)^2\hat K(z)}\nu_h(dz;\mathbb{C},\Gamma)|_{P_0}\mathbf{QC}(dh){\rm CLE}(d\Gamma)=\mathbf{QC}^{\left(\frac{1}{2}\gamma\left(\frac{\kappa'}{4}-1\right),\psi(z)\right)}(dh)\nu_0(dz;\mathbb{C},\Gamma)|_{P_0}{\rm CLE}(d\Gamma).
\end{equation*}
\end{proposition}

We should also check the averages of natural measures in Definition \ref{def3} over CLE are the same.

\begin{proposition}
Define a radially symmetric measure $m$ on $\mathbb{C}$ such that $m[B(0,r)]=|r|^d$ for every $r>0$. Then for each natural measure $\nu_0$ in Definition \ref{def3}, $\mathbb{E}\left[\nu_0(dz;\mathbb{C},\Gamma)|_{P_0}\right]$ is equal to $cm$ for a deterministic constant $c>0$.
\end{proposition}
\begin{proof}
Note that $\mathbb{E}\left[\nu_0(dz;\mathbb{C},\Gamma)|_{P_0}\right]$ is a measure supported on $\mathbb{C}$. For a conformal automorphism $\phi$ on $\mathbb{C}$, by the conformal invariance of CLE, we have $\mathbb{E}\left[\nu_0(dz;\mathbb{C},\Gamma)|_{P_0}\right]=|\phi'|^d\phi^{-1}\circ\mathbb{E}\left[\nu_0(dz;\mathbb{C},\Gamma)|_{P_0}\right]$. By taking $\phi$ as $z\mapsto e^{i\theta}z$ and $z\mapsto rz$ the result follows.
\end{proof}

The proof of uniqueness of natural measure on CLE pivotal points is still similar to the case of SLE cut point. 

\begin{proof}[Proof of Theorem \ref{thm3}, Uniqueness Part]
First suppose we are in the whole-plane $\mathbb{C}$ and given a whole-plane ${\rm CLE}_{\kappa'}$ configuration $\Gamma$. Suppose $\mu_0(dz;\mathbb{C},\Gamma)$ is a natural measure on pivotal points in Definition \ref{def3}, and $\nu(dz;\mathbb{C},\Gamma)$ is the quantum natural measure we constructed in Section \ref{CPP}. In the same setting as Section \ref{CPP}, we introduce an independent LQG field $h'$ s.t. $(\mathbb{C},0,\infty,h')$ is a quantum cone of weight $\left(4-\gamma^2\right)$ with circle average embedding. Let $\mathbb{W}_0$ be the same wedge as in Section \ref{CPP}, and parametrize its right boundary $\eta_r$ by its quantum length. According to the propositions above, we are able to quantize $\mu_0|_{P_0}$ as $\mu|_{P_0}$ as in Section \ref{ANO}.

For $t>0$, let $t'=\inf\{s\ge 0:\nu(\eta_r[0,s]\cap P_0)=t\}$, and let $X_t$ be the total $\mu$-mass of $\eta_r[0,t']\cap P_0$. Note that at time $t'$, $\eta_r$ is located at an end point of $\mathbb{W}_0$, and in each component of $\mathbb{W}_0$ the loop configuration in it is an conditionally independent ${\rm CLE}_{\kappa'}$ given its boundary. Suppose $\mathcal{D}_t$ is the bubble of $\mathcal{W}_0$ where $\eta_r(t')$ is located at, and let $\mathcal{D}'_t$ be the remaining-to-be-discovered domain of the ${\rm SLE}_{\kappa'}$ curve in $\mathcal{D}_t$ up to $\eta_r(t')$. We view $\mathcal{D}_t'$ as decorated by the fjords formed by the ${\rm SLE}_{\kappa'}$ exploration curve before $\eta_r(t')$. Define $\tilde{\mathcal{W}}_0^t$ as the quantum surface formed by concatenating $\mathcal{D}_t'$ and the forested quantum disks located after $\mathcal{D}_t$. Then $\tilde{\mathcal{W}}_0^t$ is a forested quantum wedge having the same law as $\tilde{\mathcal{W}}_0$ .

By the same calculation as before, under $\phi_t$ the quantized natural measure $\nu$ remains invariant. Therefore, since $\tilde{\mathcal{W}}_0^t$ determines $(X_{t+\cdot})$ in the same way as $\tilde{\mathcal{W}}_0$ determines $(X_t)$, $(X_t)$ has independent stable increments, and by the scaling invariance of quantum cone we have $X_{at}\overset{d}=a X_t$. In conclusion, $(X_t)$ is a trivial subordinator, therefore $X_t=ct$ for some deterministic constant $c\in (0,\infty)$, which means that $\mu_0|_{P_0}=c\nu_0|_{P_0}$. Since we can take any $z\in\mathbb{C}$ to play the role of $0$, we have $\mu_0|_{P_z}=c\nu_0|_{P_z}$ (the constant must be the same since the law of CLE is invariant under translation). Hence $\mu_0=c\nu_0$, and we conclude the uniqueness in the case of whole-plane.

For a general simply connected domain $D\subset\mathbb{C}$, we can still map the configuration $\Gamma_D$ into a bubble of a whole-plane CLE configuration $\Gamma$ by a conformal map $\phi_D$, and we get a new configuration $\Gamma'$. Thanks to the locality, we can see that $\mu_0(dz;D,\Gamma_D)=|\phi_D'|^{-d}\mu_0(dz;\mathbb{C},\Gamma') $. Then the uniqueness in the whole-plane case implies the uniqueness in a general domain.
\end{proof}

\subsection{Application: the Pivotal Measure of Continuum Percolation}\label{PC1}

In this subsection we briefly give an application of Theorem \ref{thm3} to the pivotal measure of planar continuum percolation. Consider the Bernoulli critical percolation on $\delta\mathbb{T}$ and let $\alpha_4^\delta(r,R)$ be the probability that there are four arms with alternative colors connecting the two boundary pieces of the annulus $A(r,R)=\{z:r<|z|<R\}$. It is well known that $\alpha_4^\delta(\delta,1)=\delta^{5/4+o(1)}$ , and recently in \cite{DGLZ} the authors show $\alpha_4^\delta(\delta,1)=c\delta^{5/4}(1+o(1))$ for some $c>0$. Denote $\mathcal{P}^\varepsilon(\omega_\delta)$ for the set of sites in $\omega_\delta$ which have four arms in alternative colors with length $\varepsilon$. As proved in \cite{garban2014pivotal}, the discrete pivotal measure $\mu^\varepsilon_\delta(\omega_\delta)=\sum\limits_{x\in\mathcal{P}^\varepsilon(\omega_\delta)}1_x \delta^2 \alpha_4^\delta(\delta,1)^{-1}$ converges to the continuum pivotal measure as the mesh size $\delta\to 0$. Precisely, for any $\varepsilon>0$, there is a measurable map $\mu^\varepsilon_0$ from $(\mathscr{H},d_\mathscr{H})$ into the space of finite Borel measures on $\mathbb{C}$, such that $(\omega_\delta,\mu^\varepsilon_\delta(\omega_\delta))\overset{d.}\to(\omega_\infty,\mu_0^\varepsilon(\omega_\infty))$ under the above product topology. In the rest of this subsection we are in the coupling that $\omega_\delta\to\omega_\infty$ a.s. Now let $\mu_0=\limsup_{\varepsilon\to 0}\mu_0^{\varepsilon}$. According to the explanation in Section 2.3 of \cite{garban2014pivotal}, the quad-crossing topology is equivalent to the topology of the loop ensemble, therefore we can write $\mu_0=\mu_0(\cdot;\mathbb{C},\Gamma)$ as the (measure-valued) function of the loop configuration $\Gamma$.

Now one can easily define the discrete version $\eta^\delta$ of the pseudo-interface in Section \ref{CPP} (see Chapter 4 of \cite{werner2008lectures} for further references for the discrete radial exploration). Those times when $\eta^\delta$ disconnects $0$ are denoted by $\{t_{j}^\delta\}_{j\in\mathbb{Z}}$, and the connected components of $\mathbb{C}\backslash\eta^\delta[t_{j-1}^\delta,t_j^\delta]$ containing $0$ are denoted as $D_j$ (the case $j=0$ is defined to correspond to $\gamma_0^\delta$). In time interval $[t_{j-1}^\delta,t_j^\delta]$, times of $\eta^\delta$ hitting the boundary of $D_j$ are denoted as $\{h_{j,i}^\delta\}$, and $\eta^\delta[h_{j,i-1}^\delta,h_{j,i}^\delta]$ together with $\partial D_j$ forms a \textbf{pocket}. The concatenation of those pockets over $i,j$ corresponds to the quantum wedge $\mathbb{W}_0$ in Section \ref{CPP}, and we denote the collection of points $\{\eta^\delta(h_{j,i}^\delta)\}_{j,i}$ as $P_0^\delta$, which corresponds to the macroscopic pivotal points $P_0$ defined in Section \ref{BC}. Let $P_0^\delta(r,R)=P_0^\delta\cap A(r,R)$.

Note that those $\eta^\delta(t_j^\delta)$'s which correspond to the locations of color-changing in the radial exploration have 5 arms. Since there are very few 5-arm points, most points in $P_0^\delta(r,R)$, unless it is near those $\eta^\delta(t_j^\delta)$ corresponding to color-changing, will have macroscopic 4 arms (i.e.\ these points are $r$-important in the language of \cite{garban2014pivotal}, to be precise). Moreover, for points near $\eta^\delta(t_j^\delta)$, one can change the color of $\eta^\delta(t_j^\delta)$ (note that this operation will not change the pseudo-interface and only contributes a constant factor 2 in calculating the expectation of $\mu_\delta^\varepsilon(P_0^\delta(r,R))$), after that these points have macroscopic 4 arms as well. Therefore the expectation of $\mu_\delta^\varepsilon(P_0^\delta(r,R))$ is bounded uniformly over $\varepsilon<r$ and $\delta\in(0,\varepsilon)$. Passing to the scaling limit, we conclude that $\mu_0(P_0(r,R))$ has finite expectation mudulo some justifications as $P_0^\delta(r,R)$ is a random set.

\begin{lemma}
Suppose $z^\eta$ and $z$ are random points in $\mathcal{P}^\varepsilon(\omega_\delta)$ and $\mathcal{P}^\varepsilon(\omega_\infty)$ respectively, and $\lim\limits_{\eta\to 0}z^\eta=z$ in probability. Then the indicator function of $\{z^\delta\in P_0^\delta(r,R)\}$ converges in probability to the indicator function $\{z\in P_0(r,R)\}$.
\end{lemma}
\begin{proof}
Similar to the proof of Lemma 6.18 in \cite{holden2021convergence}, one can show that for the loop ensembles $\hat{\Gamma}^\delta,\hat\Gamma$ obtained after flipping the color of $z^\delta$ and $z$, $\lim_{\delta\to 0}\hat\Gamma^\delta=\hat\Gamma$ in probability in the loop space $\mathcal{L}(A)$, which in particular implies this lemma.
\end{proof}

\begin{proposition}\label{QAD}
Almost surely the $\mu_0^\varepsilon$-mass on $P_0(r,R)$ equals $\lim_{\delta\to 0}\mu_\delta^\varepsilon(P_0^\delta(r,R))$. In particular, the expectation $\mathbb{E}\left[\mu_0(P_0(r,R))\right]$ is equal to a finite number $f(r,R)<\infty$.
\end{proposition}
\begin{proof}
Suppose that $z^\delta$ is sampled from $\mu^\varepsilon_\delta|_{P_0^\delta(r,R)}$. By the a.s. convergence of $\mu^\varepsilon_\delta$ we can assume that $z^\delta$ converges a.s.\ to a random point $z\in P_0(r,R)$. By the above lemma, we can conclude that the conditional expectation $\mathbb{E}\left[f(z^\delta)\textbf{1}_{\{z^\delta\in P_0^\delta(r,R)\}}|\omega_\delta\right]$ a.s.\ converges to $\mathbb{E}\left[f(z)\textbf{1}_{\{z\in P_0(r,R)\}}|\omega_\infty\right]$ for any bounded continuous function $f$. Then the first claim follows. The second claim follows immediately from Fatou's lemma.
\end{proof}

Note that the conformal coordinate change formula of $\mu_0$ is proved in Theorem 1.1 in \cite{garban2014pivotal}. By scaling we have $f(r,r^2)=r^{\frac{3}{4}}f(1,r)$, therefore
\begin{equation*}
\mu_0(P_0\cap B(1))=\sum_{n\ge 0}\mu_0(P_0(r^{n+1},r^n))=\frac{f(r,1)}{1-r^{3/4}}<\infty,
\end{equation*}
and it exactly gives the finiteness property we want. And the locality for $\mu_0$ is trivial. Therefore $\mu_0$ satisfies the three properties in Definition \ref{def3}, i.e. it is a natural measure on ${\rm CLE}_6$ pivotal points.

Theorem \ref{thm3} then implies that the measure $\mu_0$ obtained by the discrete scaling limit is the same as the measure $\nu_0$ we constructed in Section \ref{CPP}. We remark that it is the Euclidean version of the result of Proposition 5.1 in \cite{holden2021convergence}.
 
 \section{SLE Cut Points}

Recall the setup of Theorem \ref{thm0}. That is, we are given an ${\rm SLE}_{\kappa'}$ curve $\eta$ and its cut point set $\mathcal{C}=\eta_l\cap\eta_r$. Throughout this section, we set $\kappa\in(2,4)$, $\kappa'=16/\kappa\in(4,8)$, $\gamma=\sqrt{\kappa}$ and $Q=\gamma/2+2/\gamma$, and denote $d=3-\frac{3\kappa'}{8}$ for the Hausdorff dimension of cut point set of $\eta$. The argument in this section is similar to Section \ref{PVT}, so we will be brief.

\subsection{Construction of the Natural Measure}

We first introduce an independent random distribution $h'$ with circle average embedding on $(\mathbb{H},0,\infty)$ such that $(\mathbb{H},h',0,\infty)$ is a quantum wedge of weight $(\frac{3}{2} \gamma^2 -2)$, and parametrize $\eta_l$ by its quantum length. According to Theorem \ref{QNT}, the quantum surface between $\eta_l$ and $\eta_r$ is a quantum wedge of weight $(2-\gamma^2/2)$.

\begin{proposition}
$\eta_l^{-1}(\mathcal{C})$ could be realized as the range of a $(2-\kappa'/4)$-stable subordinator. Therefore its $(2-\kappa'/4)$-occupation measure exists, which we will denote by $m$. 
\end{proposition}
\begin{proof}
The proof is identical to Proposition \ref{SUB}. Just take $W=2-\gamma^2/2$ and we get the result.
\end{proof}

Then we consider the pushforward under $\eta_l$ of this occupation measure $m$, which we denote by $\nu_{h'}(dz;\mathbb{H},\eta)$. We still take a Dirichlet GFF $h_0$. Let $\nu_0(dz;\mathbb{H},\eta)=e^{-\frac{1}{2}(\gamma-\frac{2}{\gamma})^2\hat K_0(z)}\textbf{E}\nu_{h_0}(\cdot;\mathbb{H},\eta)$, where $\hat K(z)=\lim\limits_{\varepsilon\to 0}\log\varepsilon+\mathrm{Var}(\theta_z^\varepsilon,h_0)$. In parallel to Proposition \ref{FI}, we can show that this expectation is finite. For simply connnected domain $(D,a,b)$, one can similarly define $\nu_0(dz;D,\eta)$. The proof of locality for $\nu_0$ is same as Proposition \ref{LOCAL}. The conformal Markov property in Definition \ref{def0} follows from a direct calculation as in Proposition \ref{CFM}, where we will make repeated use of the KPZ relation $\frac{1}{2}(\gamma-\frac{2}{\gamma})^2-(\gamma-\frac{2}{\gamma})Q+d=0$ for the SLE cut point case.

\subsection{Resampling Identity}

WLOG we are now in the upper half-plane, and we denote $\nu_{\cdot}(dz,\eta)=\nu_{\cdot}(dz;\mathbb{H},\eta)$ for simplicity. For $\nu_0(dz;\eta)$ and $\nu_h(dz;\eta)$ constructed in the above subsection, we state the following resampling identity at first. We omit the proof since it is identical to Proposition \ref{IED}.

\begin{proposition}
For $\nu_0$ and $\nu$ in the above section, we have the resampling identity
\begin{equation*}
e^{-\frac{1}{2}(\gamma-\frac{2}{\gamma})^2\hat K(z)}\nu_h(dz;\eta)\mathbf{QW}_{\mathbb{H}}(dh){\rm SLE}(d\eta)=\nu_h(dz;D,\Gamma)\mathbf{QW}_{\mathbb{H}}^{(\gamma-\frac{2}{\gamma},z)}(dh)\nu_0(dz;\eta){\rm SLE}(d\eta).
\end{equation*}
\end{proposition}

In order to sample an ${\rm SLE}_{\kappa'}$ curve $\eta$, one can sample its left boundary $\eta_l$ by its marginal law, then sample $\eta_r$ conditioned on $\eta_l$, next sample curves independently in each pocket formed by $\eta_l$ and $\eta_r$, and finally concatenate them to form $\eta$. This sampling process can be written as ${\rm SLE}(d\eta)={\rm SLE}^I(d\eta;\eta_l,\eta_r){\rm SLE}^R(d\eta_r;\eta_l){\rm SLE}^L(d\eta_l)$. The following proposition checks the uniqueness of the average of a natural measure in Definition \ref{def0} over ${\rm SLE}^I(d\eta;\eta_l,\eta_r){\rm SLE}^R(d\eta_r;\eta_l)$.

\begin{proposition}
For each natural measure $\mu_0$ in Definition \ref{def0}, the average of $\mu_0(dz;\eta)$ over ${\rm SLE}^I(d\eta;\eta_l,\eta_r){\rm SLE}^R(d\eta_r;\eta_l)$ (namely, the conditional expectation ${\rm SLE}[\mu_0(dz;\eta)|\eta_l]$) is unique up to a constant (which might depends on $\eta_l$).
\end{proposition}
\begin{proof}
Suppose that we have sampled $\eta_l$, and we parametrize $\eta_l$ by its Minkowski content. Fix $r>0$. When scaling the space by $\phi:z\mapsto rz$, denote $\eta_l'$ for the image of $\eta_l$, therefore $\eta_l[0,1]=\phi^{-1}\circ\eta_l'[0,r^{d_0}]$, where $d_0$ is the Hausdorff dimension of $\eta_l$. Therefore, according to the scaling property for $\mu_0$, we have $\mu_0(\eta_l[0,1];\eta)=\mu_0(\phi^{-1}\circ\eta_l'[0,r^{d_0}];\eta)=\mu_0(\eta_l'[0,r^{d_0}];\eta')r^{-d}$. In particular, since $\sigma(\eta_l)=\sigma(\eta_l')$, we have the conditional expectation $\mathbb{E}\left[\mu_0(\eta_l[0,r];\eta)|\eta_l\right]\overset{d}=r^{d/d_0}\mathbb{E}\left[\mu_0(\eta_l[0,1];\eta)|\eta_l\right]$. Then for any ball $B\subset\mathbb{H}$, let $\{[\sigma_i,\tau_i]\}_{i\in\mathbb{N}}$ be those segments for $\eta_l$ to be in $B$, since those $[\sigma_i,\tau_i]$ are measurable with respect to $\eta_l$, we have $\mathbb{E}[\mu_0(B;\eta)|\eta_l]=\mathbb{E}\left[\sum_{i}\mu_0(\eta_l[\sigma_i,\tau_i];\eta)|\eta_l\right]=\sum_{i}(\tau_i^{d/d_0}-\sigma_i^{d/d_0})\mathbb{E}\left[\mu_0(\eta_l[0,1];\eta)|\eta_l\right]$. This means when conditioned on $\eta_l$ the measures $\mathbb{E}[\mu_0(B;\eta)|\eta_l]$ are the same for all natural measures in Definition \ref{def0} up to a constant depending on $\eta_l$.
\end{proof}

Therefore note that our resampling identity in Proposition \ref{IED} can be rewritten as
\begin{equation*}
\begin{aligned}
\textbf{QW}_{\mathbb{H}}^{(\gamma-\frac{2}{\gamma},z)}(dh)&\nu_0(dz;\eta){\rm SLE}^I(d\eta;\eta_l,\eta_r){\rm SLE}^R(d\eta_r;\eta_l){\rm SLE}^L(d\eta_l)
\\=e^{-\frac{1}{2}(\gamma-\frac{2}{\gamma})^2\hat K(z)}&\nu_h(dz;\eta){\rm SLE}^I(d\eta;\eta_l,\eta_r){\rm SLE}^R(d\eta_r;\eta_l)\textbf{QW}_{\mathbb{H}}(dh){\rm SLE}^L(d\eta_l),
\end{aligned}
\end{equation*}
We conclude that our natural measure satisfies the uniqueness assumptions in Section \ref{ANO}. Hence we can construct the quantized version $\nu_h$ of natural measure $\nu_0$ in Definition \ref{def0}.

\subsection{Uniqueness}
We now prove the uniqueness part of Theorem \ref{thm0}. WLOG we are in the upper half-plane. We introduce an independent $h$ on $(\mathbb{H},0,\infty)$ with circle average embedding such that $(\mathbb{H},h,0,\infty)$ is a quantum wedge of weight $(\frac{3}{2} \gamma^2 -2)$, and suppose that there is a natural measure $\mu_0=\mu_0(dz;\eta)$ on SLE cut points in Definition \ref{def0}. According to the above subsection and Section \ref{ANO}, we can get its quantized measure $\mu=\mu_{h}(dz,\eta)$.

We first define $D^m$ as the region bounded by $\eta_l$ and $\eta_r$, and $D^l$, $D^r$ as the interior of the left and right connected components of $\mathbb{H}\backslash D^m$ respectively. Recall the occupation measure $m$ defined in Proposition \ref{SUB}, and take $t'=\inf\{s\ge 0: m[0,s]=t\}$. Now let $D_t^m$ (resp. $\tilde D_t^m$) be the closure of the unbounded (resp. bounded) component of $\overline D_m \backslash \{\eta_l(t')\}$, $\mathcal{W}_t^m=(D_t^m,h',\eta_l(t'),\infty)$, $\mathcal{W}_t^l=(D_l,h',\eta_l(t'),\infty)$ and $\mathcal{W}_t^r=(D_r,h',\eta_l(t'),\infty)$. We now define the random variable $X_t=\nu(\tilde D_t^m)$.

According to Theorem \ref{QNT}, we find $(\mathcal{W}_0^l,\mathcal{W}_0^m,\mathcal{W}_0^r)$ are quantum wedges of weights $(\gamma^2-2)$, $(2-\gamma^2/2)$ and $(\gamma^2-2)$ respectively, and $\mathcal{W}_t^m\overset{d}=\mathcal{W}_0^m$. Let $\psi_t$ be the conformal transformation from $\mathbb{H}$ to $\mathbb{H}\backslash \tilde D_t^m$, mapping $0$ to $\eta_l(t')$ such that $h_t'=h'\circ\psi_t+Q\log|\psi_t'|$ is in its circle average embedding. Similar to the counterpart calculation in Section \ref{UNS}, one can show $\nu_{h_t'}(\psi_t(dz))=\nu_{h'}(dz)$.

Now we note that near $\eta_l(t')$, according to the relation between quantum disk and the quantum half-plane, the local picture (i.e.\ field $h'$ and $\eta$ restricted in the $\delta$-neighbor of $\eta_l(t')$) is the conformal welding of quantum wedges of weights $2$, $(2-\gamma^2/2)$ and $2$ respectively, where the welding curve is ${\rm SLE}_{\kappa'}(\kappa'-4)$. To be more explicit, the former $\delta$-local picture together with quantized measure $\nu_{h'}$ is absolutely continuous w.r.t. the latter, and the Radon-Nikodym derivative tends to $1$ when $\delta\to0$. When exploring in this local picture as before, by scaling we observe that the analog of $(X_t)$ is $Y^{t,\delta}_\cdot=\frac{X_{t+\delta\cdot}-X_t}{\delta}$. In particular, we find the limit process $Y_\cdot^t=\lim_{\delta\to0}{\frac{X_{t+\delta\cdot}-X_t}{\delta}}$ (the limit is in law) exists, and using the same argument as in Section \ref{UNS}, it has stationary and independent increments and satisfies $Y_s^t=as$ for some constant $a$ and all $t,s>0$.

Back to the process $(X_t)$, we already know that its path is continuous and non-decreasing. By adding-constant invariance of quantum wedges, we have $(X_t\cdot)\overset{d}=t(X_\cdot)$. The above analysis of $Y$ shows that $\lim_{\delta\to0} \frac{X_{t+\delta}-X_t}{\delta}=a$ (the limit is in probability) for each $t$;. By taking the filtration $\mathcal{F}_t$ generated by $\eta_l$ and $\eta_r$ up to $\eta_l(t')$ together with $h$ restricted in those components which have been finished before $\eta_l(t')$, we have $(X_t)$ is adapted and has independent increments w.r.t. $\mathcal{F}_t$.

Therefore, fix $t_0>0$ and $\lambda>0$. For any $\varepsilon>0$, by bounded convergence, there exists $\delta_0>0$ such that $\mathbb{E}e^{-\lambda\frac{X_{t_0+\delta}-X_{t_0}}{\delta}}\le e^{-\lambda(a-\varepsilon)}$ for all $\delta<\delta_0$. For $t>t_0$, by scaling we have $\mathbb{E}e^{-\lambda\frac{X_{t+\delta}-X_{t}}{\delta}}=\mathbb{E}e^{-\lambda\frac{X_{t_0+\delta'}-X_{t_0}}{\delta'}}\le e^{-\lambda(a-\varepsilon)}$ , where $\delta'=\delta\frac{t_0}{t}<\delta$. Hence by independence, for all $N>\frac{t-t_0}{\delta_0}$ and $\delta=\frac{t-t_0}{N}$, $\mathbb{E}e^{-\frac{\lambda}{\delta}(X_t-X_{t_0})}\le e^{-\frac{\lambda}{\delta}(a-\varepsilon)(t-t_0)}$ , i.e. $\mathbb{E}e^{-\frac{\lambda}{\delta}[(X_t-X_{t_0})-(a-\varepsilon)(t-t_0)]}\le 1$ for any $t>t_0$, $\varepsilon>0$ and sufficiently small $\delta$. This implies almost surely $(X_t-X_{t_0})-(a-\varepsilon)(t-t_0)\ge 0$. Similarly we know almost surely $(X_t-X_{t_0})-(a+\varepsilon)(t-t_0)\le 0$. Since $\varepsilon>0$ is arbitrary, we have $X_t-X_{t_0}=a(t-t_0)$. By path continuity we conclude that $X_t=at$ for every $t$ almost surely. The uniqueness then follows.

\begin{remark}
Using the argument in the above proof, we can also discuss the natural measure of boundary intersecting points of an ${\rm SLE}_{\kappa'}$ process for $\kappa'\in(4,8)$ in the context of Section \ref{UNS}. Note that in \cite{bound2009} there is already an axiomatic construction of this natural measure. The LQG method above could give this a new construction and a new proof of uniqueness.
\end{remark}

\section{CLE Carpet}\label{CPT}

Recall the setup of Theorem \ref{thm1}. Throughout this section, consider ${\rm CLE}_\kappa$ configuration on a simply connected domain $D$ where $\kappa\in(\frac{8}{3},4)$. We let $\kappa'=16/\kappa$, $\gamma=\sqrt{\kappa}$, $\alpha=\frac{4}{\kappa}$, $\beta=\alpha+\frac{1}{2}$, $Q=\frac{\gamma}{2}+\frac{2}{\gamma}$ and let $d=2-\frac{(3\kappa-8)(8-\kappa)}{32\kappa}$  be the Hausdorff dimension of the CLE carpet. In this section a crucial tool is the Corformal percolation interface introduced in \cite{MILLER_2017} of CLE and its coupling with an independent LQG background.

\subsection{Conformal Percolation Interface}\label{CPI}

For a ${\rm CLE}_\kappa$ configuration $\Gamma$ in a simply connected domain $D$ with two marked points $x$ and $y$ on $\partial D$, as proved in Proposition 4.1 of \cite{MILLER_2017}, by adding some randomness, it is possible to make sense of a non-self-crossing curve from $x$ to $y$, which is called the Corformal Percolation Interface (CPI), such that it always stays in the CLE carpet and always leaves a CLE loop to its right if it hits it.

Suppose we have a CPI curve $\eta$ of configuration $\Gamma$ from $x$ to $y$. Then let $\eta_t^\star$ be the set consisting of $\eta[0,t]$ together with all the loops of $\Gamma$ it intersects, and let $\mathcal{F}_t^\star$ be the $\sigma$-algebra generated by $(\eta_s^\star,\eta_s)_{s\le t}$. Denote $D_t^0$ for the set obtained by removing $\eta_t^\star$ and all the interiors of the loops from $D$, and $D_t$ be the connected component of $D_t^0$ which contains $y$. Then for any $(\mathcal{F}_t^\star)$-stopping time $\tau$, according to Definition 2.1 of \cite{MILLER_2017}, we have the following conformal Markovian property of the CPI curve $\eta$:

\begin{enumerate}[1)]
\item given $\mathcal{F}_\tau^\star$, the conditional law of $\phi_\tau(\eta, \Gamma)|_{D_\tau}$ is equal to the joint law $(\gamma,\Gamma)$, where $\phi_t$ is the conformal transform from $D_t$ to $D$ preserving $x$ and $y$, and $\phi_t(z)\sim z$ in the neighborhood of $y$;
\item the conditional law of $\Gamma$ in other connected components of $D_\tau^0$ is an independent CLE.
\end{enumerate}


The \textit{annealed} law of CPI is an ${\rm SLE}_{\kappa'}(0,\kappa'-6)$ process. When the CPI disconnects the remaining-to-be-explored domain into two pieces, we call it \textit{cuts out} a surface (corresponding to the piece whose boundary does not contain $y$) at this time. We can iteratively draw a CPI curve in each cut-out surface.

Now, if we add an independent LQG background $h'$ and parameterize the CPI by its quantum length, then according to Theorem 1.4 in \cite{miller2020simple}, in the case that $(D,h')$ is a realization of the \textbf{quantum half-plane}, and the ordered family of cut-out surfaces is a Poisson point process of \textbf{quantum disks}. As for the case that $(D,h')$ is a quantum disk, those cut-out quantum surfaces are \textit{still} independent quantum disks conditional on its boundary length, which can be shown by using some absolute continuity arguments (see Theorem 1.1 and the proof of Proposition 5.1 in \cite{miller2020simple}).

\subsection{Existence of the CLE carpet measure}\label{EX}

We add an independent LQG background $h'$ on $D$ which has the law of a quantum disk with unit boundary length. The quantum natural measure on CLE carpet has been constructed in Theorem 1.3 of \cite{miller2020simple}. More precisely, for any given open set $O$, defining $N_\varepsilon(O)$ as the number of ${\rm CLE}_\kappa$ loops entirely in $O$ with quantum length greater than $\varepsilon$, then as $\varepsilon\to 0$, $\varepsilon^{\alpha+1/2}N_\varepsilon(O)$ converges in probability to a non-trivial finite random variable $\mathcal{Y}_{h'}(O)$, which is defined as the quantum natural measure of $O$.

We still average over a Dirichlet GFF $h_0$ to get the natural measure
\begin{equation}
\mathcal{Y}_0(dx)=e^{-\frac{1}{8}\gamma^2(\alpha+\frac{1}{2})^2 \hat K_0(x)}\textbf{E}\mathcal{Y}_{h_0}(dx).
\end{equation}
According to Lemma 4.4 in \cite{ang2021integrability}, the moment $\mathbb{E}[\mathcal{Y}(D)^p]$ is finite for all $p\in(0,1+1/\beta)$. This verifies the above definition and the finite expectation condition in Definition \ref{def1}. It is easy to check that this measure satisfies remaining two conditions in Definition \ref{def1}.

\begin{itemize}
\item The conformal coordinate change formula of $\mathcal{Y}_0(\cdot;D,\Gamma)$ can be checked as Section \ref{CPP}. 
\item The locality property can be proved along the same line as Proposition \ref{LOCAL}. 
\end{itemize}

Therefore, we conclude that $\mathcal{Y}_0(dx;D,\Gamma)$ is a natural measure on the CLE carpet.

\subsection{Uniqueness}\label{UN}

As before we need to check those three assumptions in Section \ref{ANO}. Its resampling identity are showed in Proposition 4.10 in \cite{ang2021integrability}, while the uniqueness of $\mathbb{E}[\mathcal{Y_0}(dz;D,\Gamma)]$ are showed in Lemma 5.1 in \cite{miller2021cle}. Therefore, suppose we are given a natural measure $\nu_0(\cdot; D,\Gamma)$ on CLE carpet in Definition \ref{def1} and a independent quantum half-plane $h'$(with circle average embedding, for example), we can construct $\nu=\nu_{h'}(dz; D,\Gamma)$ the LQG tilting measure of $\nu_0$. By conformal covariance of $\nu_0$, we could check that this $\nu_{h'}(dz; D,\Gamma)$ does not depend on the choice of representatives among the equivalence class of quantum surfaces. Then we denote the total $\nu$-mass of an independent quantum disk with an independent CLE configuration conditioned on its boundary length being $\delta$ by $m^\nu_\delta$. Note that $m^\nu_\delta$ has the same law as $\delta^{\alpha+1/2}m^\nu_1$.

We first explore in the quantum half-plane by the CPI curve $\eta$. We parametrize $\eta$ by its quantum natural time, and denote the sequence of quantum surfaces cut out by the trunk with $(S_n)_{n\in \mathbb{Z}}$ (in time order). Define $L_t,R_t$ as the change in the boundary length of the left or right side of the unbounded component of $\mathbb{H}\backslash\eta[0,t]$ relative to the boundary lengths at time 0. As explained in Corollary 3.2 of \cite{miller2021nonsimple}, $L$ and $R$ are independent $\alpha$-stable Levy processes. 

\begin{proposition}
Denote $X_t$ for the total $\nu$-mass of quantum surfaces that have been cut out before time $t$. Then the following properties hold for $X$.
\begin{itemize}
\item  The jump time of $L$ or $R$ is also the jump time of $X_t$.
\item  Conditioning on the jump size $\delta$ of $(L,R)$, the jump size of $X_t$ equals the total $\nu$-mass of an independent quantum disk with an independent CLE configuration conditional on its boundary length being $\delta$.
\item  $(X_t)$ has independent and stationary increments, with $X_{at}\overset{d}=a^{1+\frac{\alpha}{2}}X_t$ for every $t>0$ (i.e, $(X_t)$ is a stable subordinator with index $\left(1+\frac{\alpha}{2}\right)^{-1}$).
\end{itemize}
\end{proposition}
\begin{proof}
The first and second claims are quite straightforward. For the third claim, note that $(S_n)_{n\in\mathbb{Z}}$ is a Poisson point process of quantum disks. Observe that the field $h'$ and the configuration $\Gamma$ form conditionally independent CLE configurations on quantum disks, and the natural measure $\nu_0$ has locality property, as well as by Proposition \ref{RST}, $\nu$ restricted on $S_n$ only depends on $h'|_{S_n}$. Hence $\nu_{h'}(dz; S_n,\Gamma)$ is exactly $\nu_{h'}(dz; D,\Gamma)$ restricted on $S_n$, and they form independent random measures conditioning on the CPI $\eta$. Therefore, we find that $(X_t)$ has independent and stable increments. If we add a constant $C$ to $h'$, then the quantum natural time and total mass will be scaled by $e^{\frac{2}{\gamma}C}$ and $e^{\frac{\gamma}{2}\left(\alpha+\frac{1}{2}\right)}$ respectively. Therefore by scaling we can see $X_{at}\overset{d}=a^{1+\frac{\alpha}{2}}X_t$ for every $t>0$.
\end{proof}

Now suppose there is another natural measure $\mu_0$ in Definition \ref{def1}, and let $\mu=\mu_{h'}(dz; D,\Gamma)$ be its quantization in the same way as in Section \ref{ANO}. We claim that their corresponding explorations have the same law.

\begin{proposition}\label{EQU}
The law of $(L_t,R_t,X_t)$ is unique up to a multiplicative constant in the third coordinate. That is, for any other $\mu$ mentioned above, which corresponds to the triple $(L_t,R_t,X_t')$ (i.e., $X_t'$  is the sum of $\mu$-mass of quantum surfaces that have been cut out before time $t$), there is a deterministic constant $c$ such that $(L_t,R_t,X_t)\overset{d}=(L_t,R_t,cX_t')$. In particular $m_1^\nu\overset{d}=cm_1^\mu$ for the same $c$.
\end{proposition}
\begin{proof}
By the above proposition we see that $(X_t)$ and $(X_t')$ are all stable subordinators with index $\left(1+\frac{\alpha}{2}\right)^{-1}$. By Proposition \ref{LT}, we have $(X_t)\overset{d}=c(X_t')$ for some constant $c$. On the other hand, we see that the jump size of $X_t$ (resp.\ $X_t'$) conditioned on the jump size $\delta$ of $(L,R)$ has the same law as $\delta^{\alpha+1/2}m_1^\nu$ (resp. $\delta^{\alpha+1/2}m_1^\mu$). Integrating out $\delta$ and noting $(X_t)\overset{d}=c(X_t')$, we obtain $m_1^\nu\overset{d}=cm_1^\mu$.
\end{proof}

Now we already have the uniqueness of the law of total mass of the quantum natural measure on CLE carpet in the quantum disk with unit boundary length, we need to clarify that

\begin{claim}\label{ZER}
There is no $\nu$-mass on the CPI curve $\eta$.
\end{claim}
\begin{proof}
Suppose $(D,h)$ is a quantum disk with unit boundary length. Note that almost surely, $\nu(D)\ge\sum_{n}\nu(S_n)$. However the expectation of the left hand side is $\mathbb{E}\mathcal{Y}(D)$, while the right hand side is equal to $\mathbb{E}\left[\sum_n l_n^\beta\right]\cdot\mathbb{E}\mathcal{Y}(D)$. Therefore the above inequality is indeed an equality, which implies that $\nu(\eta)=0$.
\end{proof}

Now let CPI's explore in each quantum surface of $(S_n)_{n\in\mathbb{Z}}$ (choose the target point on $\partial S_n$ in an arbitrary fixed way, e.g. targeting at the furthest point in Euclidean metric from the closing point of $S_n$). While a CPI explores in $S_m$, it will cut out new quantum surface sequences $(S_{m,n})_{n\in\mathbb{Z}}$  and also have the law of independent quantum disks decorated with CLE if conditioned on the CPI in $S_m$. Similarly, we can define $(S_{n_1,\ldots,n_k})_{n_1,\ldots,n_k\ge 0}$  for each $k\in\mathbb{Z_+}$ by induction. By Proposition \ref{EQU} we find that $\mu(S_{n_1,\ldots,n_k})\overset{d}=c\nu(S_{n_1,\ldots,n_k}), \forall n_1,\ldots,n_k\in\mathbb{Z}$.

For an open set $O\subset \mathbb{H}$, at step $k$, we denote $I_k=\{(n_1,\ldots,n_k):S_{n_1,\ldots,n_k}\cap O \neq \emptyset\}$. Now consider the sequence $\sum_{(n_1,\ldots,n_k)\in I_k}\nu(S_{n_1,\ldots,n_k})$ , obviously it is monotonically increasing with $k$ and converges to $\nu(O)$ by Claim \ref{ZER}. For $\mu$ the situation is the same. So we get that $\mu(O)\overset{d}=c\nu(O)$ for any open sets $O\subset \mathbb{H}$ (since each $S_{n_1,\ldots,n_k}$ for the same $k$ is independent, we can add up both sides of $\mu(S_{n_1,\ldots,n_k})\overset{d}=c\nu(S_{n_1,\ldots,n_k}$), then take the limit). 

Finally, according to the above CPI exploration tree structure, we can prove that these two measures $\mu$ and $c\nu$ are exactly the same by showing their mean square error is zero.

\begin{proposition}\label{MKV}
For the same constant $c$ as above, we actually have $\mu=c\nu$ a.e.
\end{proposition}
\begin{proof}
Without loss of generality, we suppose $c=1$. According to the exploration above, we only need to show $m_1^\mu=m_1^\nu$. Let $Y$ be the total mass of the natural measure constructed in Section \ref{EX} of an (independent) quantum disk conditioned on its boundary length being $1$. We still use $(S_{n_1,\ldots,n_k})$ to denote the cut-out surfaces in the $k$-th generation and denote its boundary length by $l_{n_1,\ldots,n_k}$. We would like to show that $(\mu(D)-\nu(D))^2=0$, however as the second moment is infinite, we need to do some truncation. We recall the notation that $\mathbb{E}$ means to average out all the randomness, including the CPI, CLE \textbf{and} LQG while the symbols $P,E$ denote the law of $Y$ and the expectation w.r.t. $Y$ respectively.

For any positive number $K$, by scaling, we can see that
\begin{equation*}
\mathbb{P}[\mu(S_{n_1,\ldots,n_k})\ge K l_{n_1,\ldots,n_k}^{-\beta}]=P[Y\ge K l_{n_1,\ldots,n_k}^{-2\beta}]\le K^{-1/\beta}\mathbb{E}[l_{n_1,\ldots,n_k}^2]E[Y^{1/\beta}],
\end{equation*}
(note that $1/\beta <1+1/\beta$, then $E[Y^{1/\beta}]\in(0,\infty)$ by Lemma 4.4 in \cite{ang2021integrability}). As explained in Section 6.1 in \cite{miller2020simple}, we have $\mathbb{E}[\sum_{n}l_n^2]<1$, which implies that $\mathbb{E}[\sum_{n_1,\ldots,n_k}l_{n_1,\ldots,n_k}^2]=\mathbb{E}[\sum_{n}l_n^2]^k$ has an exponential decay with $k$, hence
\begin{equation*}
\sum\limits_{k}\sum\limits_{n_1,\ldots,n_k} \mathbb{P}[\mu(S_{n_1,\ldots,n_k})\ge K l_{n_1,\ldots,n_k}^{-\beta}]<\infty.
\end{equation*}
By Borel-Cantelli Lemma, we have that with probability $1$, 
\begin{equation*}
(\mu(D)-\nu(D))^2=\left[\sum\limits_{n_1,\ldots,n_k}(\mu(S_{n_1,\ldots,n_k})\textbf{1}_{\mu(S_{n_1,\ldots,n_k})\le K l_{n_1,\ldots,n_k}^{-\beta}}-\nu(S_{n_1,\ldots,n_k})\textbf{1}_{\nu(S_{n_1,\ldots,n_k})\le K l_{n_1,\ldots,n_k}^{-\beta}})\right]^2
\end{equation*}
for sufficiently large $k$.

Note that the cross terms that will appear when we expand the above product are equal to zero. It is sufficient to show this for the cross term of two first-generation cut-out surfaces $S_i$ and $S_j$. Since the restricted configurations and their corresponding quantum surfaces are all independent with each other conditioned on the CPI curve, the locality of measures $\mu$ and $\nu$ implies that
\begin{equation*}
\mu(S_i)\textbf{1}_{\mu(S_i)\le K l_{i}^{-\beta}}-\nu(S_i)\textbf{1}_{\nu(S_i)\le K l_{i}^{-\beta}} \quad \text{and} \quad \mu(S_j)\textbf{1}_{\mu(S_j)\le K l_{j}^{-\beta}}-\nu(S_j)\textbf{1}_{\nu(S_j)\le K l_{j}^{-\beta}}
\end{equation*}
are conditionally independent of each other. As we have already seen that the law of $\mu$ and $\nu$ are the same, we get that $\mathbb{E}[\mu(S_i)\textbf{1}_{\mu(S_i)\le K l_{i}^{-\beta}}]=\mathbb{E}[\nu(S_i)\textbf{1}_{\nu(S_i)\le K l_{i}^{-\beta}}]$, thus 
\begin{equation*}
\mathbb{E}\left[\left(\mu(S_i)\textbf{1}_{\mu(S_i)\le K l_{i}^{-\beta}}-\nu(S_i)\textbf{1}_{\nu(S_i)\le K l_{i}^{-\beta}}\right)\left(\mu(S_j)\textbf{1}_{\mu(S_j)\le K l_{j}^{-\beta}}-\nu(S_j)\textbf{1}_{\nu(S_j)\le K l_{j}^{-\beta}}\right)\right]=0.
\end{equation*}

Therefore we only need to deal with the term
\begin{equation*}
\sum\limits_{n_1,\ldots,n_k}\left(\mu(S_{n_1,\ldots,n_k})\textbf{1}_{\mu(S_{n_1,\ldots,n_k})\le K l_{n_1,\ldots,n_k}^{-\beta}}-\nu(S_{n_1,\ldots,n_k})\textbf{1}_{\nu(S_{n_1,\ldots,n_k})\le K l_{n_1,\ldots,n_k}^{-\beta}}\right)^2
\end{equation*}
which is no more than 
$$2\sum\limits_{n_1,\ldots,n_k}\left[\mu(S_{n_1,\ldots,n_k})^2\textbf{1}_{\mu(S_{n_1,\ldots,n_k})\le K l_{n_1,\ldots,n_k}^{-\beta}}+\nu(S_{n_1,\ldots,n_k})^2\textbf{1}_{\nu(S_{n_1,\ldots,n_k})\le K l_{n_1,\ldots,n_k}^{-\beta}}\right].
$$
By scaling property again we can see $\nu(S_{n_1,\ldots,n_k})^2\textbf{1}_{\nu(S_{n_1,\ldots,n_k})\le K l_{n_1,\ldots,n_k}^{-\beta}}\overset{d}{=} l_{n_1,\ldots,n_k}^{2\beta} Y^2 \textbf{1}_{Y\le K l_{n_1,\ldots,n_k}^{-2\beta}}$ (conditioned on those $l_{n_1,\ldots,n_k}$). According to the exact law of $Y$ given in Section 4.1 of \cite{ang2021integrability} , the density of $Y$ is $O(1)x^{-2-1/\beta}$ , hence $E[Y^2\textbf{1}_{Y\le K l_{n_1,\ldots,n_k}^{-2\beta}}]\le C  l_{n_1,\ldots,n_k}^{2\beta -2}$. Therefore we have 
\begin{equation*}
\mathbb{E}\left[\sum\limits_{n_1,\ldots,n_k}\nu(S_{n_1,\ldots,n_k})^2\textbf{1}_{\nu(S_{n_1,\ldots,n_k})\le K}\right]\le C_K\mathbb{E}\left[\sum\limits_{n_1,\ldots,n_k}l_{n_1,\ldots,n_k}^2\right].
\end{equation*}
Using the fact that $\mathbb{E}[\sum_{n_1,\ldots,n_k}l_{n_1,\ldots,n_k}^2]$ has an exponential decay again it follows that as $k\to\infty$,
\begin{equation*}
\mathbb{E}\left[\sum\limits_{n_1,\ldots,n_k}\Big(\mu(S_{n_1,\ldots,n_k})\textbf{1}_{\mu(S_{n_1,\ldots,n_k})\ge K l_{n_1,\ldots,n_k}^{-\beta}}-\nu(S_{n_1,\ldots,n_k})\textbf{1}_{\nu(S_{n_1,\ldots,n_k})\ge K l_{n_1,\ldots,n_k}^{-\beta}}\Big)\right]^2\to 0,
\end{equation*}
therefore, the sum in the expectation converges to $0$ in probability. Combining all things together we finally have that $(\mu(D)-\nu(D))^2=0$, whence $\mu=\nu$.
\end{proof}

Since $\nu_{h'}(dz; D,\Gamma)\textbf{QD}_D(dh')=\textbf{QD}_D^{(\frac{1}{2}\gamma\beta,z)}(dh')\nu_0(dz;D,\Gamma)$, taking expectation of $\nu_{h'}(dz; D,\Gamma)$ under $\textbf{QD}_D(dh')$ gives $\nu_0(dz;D,\Gamma)$ (up to a geometric factor). Since $\mu=c\nu$, we see that $\mu_0=c\nu_0$ as well. This finishes our proof of uniqueness.

\section{CLE Gasket}\label{GAT}

Recall the set up of Theorem \ref{thm1}. Consider a ${\rm CLE}_{\kappa'}$ configuration on $D$ where $\kappa'\in(4,8)$. Let $\alpha'=4/\kappa'$, $\beta'=\alpha'+1/2$, write $d'=2-\frac{(3\kappa'-8)(8-\kappa')}{32\kappa'}$  for the Hausdorff dimension of the CLE gasket. Many arguments are parallel to those in Section \ref{CPT}, so we will be brief in this section.

\subsection{Exploration of ${\rm CLE}_{\kappa'}$}\label{EP}

We want to explore the ${\rm CLE}_{\kappa'}$ configuration like what we did with CPI in Section \ref{CPT}. The exploration process we take is indeed the inverse of our construction of CLE in Section \ref{CPP}, which is stated in Theorem 5.4 of \cite{sheffield2009exploration}.

WLOG suppose $(D,a,b)=(\mathbb{H},0,\infty)$. Let $L_1,L_2,\ldots$ be the loops in $\rm CLE_{\kappa'}$ configuration which intersect $(-\infty,0]$. For each $i\ge1$, let $I_i$ be the interval $\big(\inf L_i\cap(-\infty,0),\sup L_i\cap(-\infty,0)\big)$. Let those loops be counterclockwise oriented, and $A_i$ be the upper portion of $L_i$ from $\sup I_i$ to $\inf I_i$. We consider the concatenation path $\eta$ of $A_i$'s such that $I_i$ is not contained in any other $I_j$ (i.e., $L_i$ is the \textit{outermost} loop). Then according to Theorem 5.4 in \cite{sheffield2009exploration}, the law of $\eta$ will be a ${\rm SLE}_{\kappa'}(\kappa'-6,0)$ curve. We call this curve $\eta$ the ${\rm CLE}_{\kappa'}$ \textbf{interface}. Here we use the word \textit{interface} since its analog in discrete percolation on $\mathbb{H}$ with Dobrushin boundary condition is exactly the percolation interface.

In this section, each bounded connected component of $\mathbb{H}\backslash\eta$ which is not surrounded by any outermost CLE loop is called a (first-generation) \textbf{pocket}. 

Consider each first-generation pocket. According to the ${\rm SLE}_{\kappa'}(\kappa'-6)$ exploration tree construction of ${\rm CLE}_{\kappa'}$ in Section 4.3 of \cite{sheffield2009exploration}, the ${\rm CLE}_{\kappa'}$ configuration restricted in each pocket is a CLE conditionally independent from each other. Therefore by induction we can similarly define pockets for each generation.

Now, we introduce an independent LQG background $h'$ and parametrize each ${\rm SLE}_{\kappa'}(\kappa'-6,0)$ curve by its quantum length. By the conformal welding theorem \ref{WD}, each first-generation pocket has the law of a quantum disk when $(\mathbb{H},h',0,\infty)$ is a realization of quantum half-plane. In the case that $(\mathbb{H},h',0,\infty)$ has the law of a quantum disk, those poskets are quantum disks conditioned on their boundary lengths. These results is mentioned in the proof of Theorem 5.1 in \cite{miller2021nonsimple}, which is based on the absolute continuity between quantum disk and quantum half-plane.

\subsection{Existence of the CLE Gasket Measure}\label{EX2}

As explained in Section 5.2 in \cite{miller2021nonsimple}, by adding a LQG background to make it a quantum disk conditioned on its boundary length being $1$, one can define a quantum natural measure $\mu$ in the CLE gasket such that for a class of open sets $O$ (sufficient to generate the Borel $\sigma$-algebra in the domain), we have $\mu(O)=\lim_{\varepsilon\to 0}\varepsilon^{\alpha'+1/2}N_\varepsilon(O)$ where $N_\varepsilon(O)$ is the number of CLE loops of generalized boundary length in $[\varepsilon,2\varepsilon]$ in $O$. After the same dequantizing process as in Section \ref{EX}, we can get a dequantized measure. Its coordinate change formula and locality can be directly varified in parallel to these in Section \ref{EX}. Its finiteness of expectation is showed in the following Proposition \ref{GSK}. Therefore, it is a natural measure on the CLE gasket in Definition \ref{def1}.

\subsection{Uniqueness}

Most of this subsection are exactly parallel to the case of the CLE carpet. As in Section \ref{EP}, without loss of generality we let $(D,x,y)=(\mathbb{H},0,\infty)$. Suppose we are given a natural measure $\nu_0$ on CLE gasket. We first introduce an independent random distribution $h'$ on $(\mathbb{H},0,\infty)$ such that $(\mathbb{H},h',0,\infty)$ is a quantum half-plane. We then tilt $\nu_0$ by LQG background, that is, we let $\nu=\nu_{h'}$ be its quantization according to Section \ref{ANO} (it is essentially the same to check three assumptions as in the carpet case, so we omit it).

We can check the conformal covariance as before, then define $L_t,R_t$ as the change of the left or right boundary at time $t$ and let $X_t$ be the $\nu$-mass of all the pockets formed by $\eta$ before time $t$. By the same argument we can see it is a stable subordinator with index $\beta'^{-1}$. The counterpart of Proposition \ref{EQU} still holds, and in particular the law of the total mass of an unit boundary length quantum disk is unique up to a constant. One can also prove that the $\nu$-mass on the interface $\eta$ is $0$ in parallel to Claim \ref{ZER}.

Using the same inductive arguments as in the carpet case, we see that the law of the $\nu$-mass of pockets in all generations are determined up to a multiplicative constant. Furthermore, the law of $\nu$-mass of each bounded open set is also determined up to a multiplicative constant. 

In the end, we use truncation and the second moment argument as in Proposition \ref{MKV} to show that the natural measure are indeed determined up to a multiplicative constant. According to the last part of its Section 1 in \cite{CCM2020}, Proposition 2.21 in \cite{ang2021integrability} still holds for non-simple CLE's. By the same argument used in the proof of Lemma 4.3 in \cite{ang2021integrability}, we have the following proposition.

\begin{proposition}\label{GSK}
For a quantum disk conditioned on its boundary length being $1$ decorated with an independent ${\rm CLE}_{\kappa'}$ configuration, denote $Y'$ for its total mass of the natural measure that has been constructed in Section \ref{EX2}. Let $(\zeta_t)_{t\ge 0}$ be a $\beta'$-stable Levy process whose Levy measure is $\textbf{1}_{x>0}x^{-\beta'-1}dx$ so that it has no downward jumps, and denote its law by $P^{\beta'}$. Let $\tau_{-a}=\inf\{t:\zeta_t=-a\}$. Then the law of $Y'$ is the same as the law of $\tau_{-1}$ under $\frac{\tau_{-1}P^{\beta'}}{E^{\beta'}[\tau_{-1}]}$. In particular, the tail of $Y'$ is $O(1)x^{-2-1/\beta'}$.
\end{proposition}

Therefore, we can use a similar argument to the proof of Proposition \ref{MKV} (just add a prime to some quantities) to show that two natural measures $\mu$ and $\nu$ must be the same.

\subsection{Application: the Area Measure of Continuum Percolation}\label{PC2}

Consider the Bernoulli critical percolation on $\delta\mathbb{T}$, and let $\alpha_1^\delta(r,R)$ be the probability that there is one arm connecting the two boundary pieces of the annulus $A(r,R)=\{z:r<|z|<R\}$. According to Theorem 5.1 in \cite{garban2014pivotal}, for a annulus $A\subset\mathbb{C}$ with piecewise smooth boundary, denotes its inner face by $I(A)$, the configuration $\omega_\delta$ together with its discrete area measure $\lambda^A_\delta(\omega_\delta)=\sum\limits_{x\in I(A):x\leftrightarrow \partial_2 A}1_x \delta^2 \alpha_1^\delta(\delta,1)^{-1}$ converges to the continuum area measure $(\omega_\infty,\lambda_0^A)$ as the mesh size $\delta\to 0$. In our case, let $D$ be a $C^1$ topological quadrangle, and $(A_n)$ be a sequence of decreasing annulus with $\partial_2 A_n=\partial D$ and $\partial_1 A_n\to\partial D$ as $n\to\infty$. Clearly these measures are compatible on their common support, thus we can denote $\lambda_0(\omega_\infty)$ for the collection of measures $\lambda_0^{A_n}(\omega_\infty)$. For the same reason as in Section \ref{PC1} we can write the measure $\lambda_0(\omega_\infty)$ as $\lambda_0(\cdot;D,\Gamma)$, where $\Gamma$ is the loop configuration equivalent to $\omega_\infty$. The conformal coordinate change formula for $\lambda_0$ follows from Theorem 6.7 in \cite{garban2014pivotal}, and its locality is obvious. The finiteness for $\mathbb{E}[\lambda_0(K;A,\Gamma)]$ for any compact $K\subset A$ is also easy. Therefore this $\lambda_0(\cdot;D,\Gamma)$ is a natural measure on ${\rm CLE}_6$ gasket in Definition \ref{def2}, by Theorem \ref{thm2} it must be equal to the ${\rm CLE}_6$ gasket measure we constructed in Section \ref{EX2}.

Furthermore, we mention that for this $\lambda_0$ we can verify the finiteness of $d$-energy. Indeed, since $\alpha_1^\delta(\delta,1)=\delta^{\frac{5}{48}+o(1)}$ and the quasi-multiplicativity (see Section 2.1 of \cite{garban2014pivotal})
\begin{equation*}
c_k\alpha_1^\delta(r_1,r_2)\alpha_1^\delta(r_2,r_3)\le\alpha_1^\delta(r_1,r_3)\le\alpha_1^\delta(r_1,r_2)\alpha_1^\delta(r_2,r_3)
\end{equation*}
for some absolute constant $c_k$, and the dimension of ${\rm CLE}_6$ gasket is $d=\frac{91}{48}$, for each $A_n$ we can first calculate the expectation of energy of $\lambda_\delta^{A_n}$ (denote $r=\text{diam}(A_n)$ and $r'=\text{dist}(\partial_1 A_n,\partial D)$)
\begin{equation*}
\begin{aligned}
\sum_{x,y\in I(A_n)\atop x,y\leftrightarrow\partial D}\frac{\delta^4\alpha_1^\delta(\delta,1)^{-2}}{|x-y|^{d-\varepsilon}}&\asymp\sum_{k=0}^{\log_2(r/\delta)}\frac{\alpha_1^\delta(\delta,2^k\delta)^2\alpha_1^\delta(2^k\delta,r')}{\delta^{-4}\alpha_1^\delta(\delta,1)^2}\frac{2^{2k}(r/\delta)^2}{(2^k\delta)^{d-\varepsilon}}\\
&\asymp\frac{r^2\delta^2}{\alpha_1^\delta(\delta,1)\alpha_1^\delta(r',1)}\sum_{k=0}^{\log_2(r/\delta)}\frac{2^{2k}\alpha_1^\delta(\delta,2^k\delta)}{(2^k\delta)^{d-\varepsilon}}\\
&\asymp\frac{r^{4-d+\varepsilon}}{\alpha_1^\delta(r,1)\alpha_1^\delta(r',1)}\le C<\infty
\end{aligned}
\end{equation*}
where $C$ is a constant that only depends on $A_n$. Therefore by Fatou's Lemma it follows that
\begin{equation*}
\mathbb{E}\left[\iint_{A_n\times A_n}\frac{\lambda_0(dx)\lambda_0(dy)}{|x-y|^{d-\varepsilon}}\right]\le C,
\end{equation*}
and we conclude the finiteness of $d$-dimension energy for $\lambda_0$ since $A_n$ exhausts the domain $D$.

\begin{remark}
Contrary to the case of $\rm{CLE}_6$ pivotal and gasket measure, in general, when there is no explicit construction of a natural measure (e.g., scaling limit of discrete models or Minkowski content), it seems hard to show the finiteness of energy for the natural measure we have constructed directly. 
\end{remark}

\end{document}